\documentclass[letterpaper, 10 pt, conference]{ieeeconf}  

\usepackage{booktabs}
\usepackage{adjustbox}

\usepackage[top=20.2mm,bottom=18mm,left=16.9mm,right=20.2mm]{geometry}

\usepackage{wrapfig}
\usepackage{lipsum}
\usepackage{amssymb}
\usepackage{amsmath}
\usepackage{siunitx}
\usepackage{mathtools}
\usepackage{hyperref}
\let\labelindent\relax
\usepackage[inline]{enumitem}
\usepackage[ruled,vlined,linesnumbered]{algorithm2e}
\usepackage[noend]{algpseudocode}
\usepackage{ifthen}
\usepackage{dsfont}
\usepackage[backend=biber,style=numeric-comp,sorting=none,maxbibnames=99]{biblatex}
\usepackage{balance}
\usepackage{dirtytalk}
\usepackage{afterpage}
\usepackage{changepage}
\usepackage{bm}

\usepackage{times}

\usepackage{svg}

\allowdisplaybreaks

\newcommand{\diag}{\text{diag}}

\newcommand{\E}{\mathbb{E}}

\newcommand{\G}{\mathcal{G}}

\newcommand{\N}{\mathbb{N}}

\newcommand{\ra}{\rightarrow}
\newcommand{\R}{\mathbb{R}}
\newcommand{\Ra}{\Rightarrow}

\newcommand{\boldQ}{\boldsymbol{Q}}
\newcommand{\boldR}{\boldsymbol{R}}

\newcommand{\boldgamma}{\boldsymbol{\gamma}}
\newcommand{\cov}{\text{Cov}}

\newcommand{\normalVec}{\mathcal{N}}

\newcommand{\vstack}{\text{vec}}

\newcommand{\tr}{\text{tr}}

\newcommand{\steady}{\text{ss}}

\newcommand{\sym}{\mathbb{S}}
\newcommand{\symPSD}{\mathbb{S}_+}
\newcommand{\symPD}{\mathbb{S}_{++}}

\newcommand{\info}{\mathcal{I}}

\newtheorem{assumption}{Assumption}
\newtheorem{proposition}{Proposition}
\newtheorem{definition}{Definition}
\newtheorem{lemma}{Lemma}
\newtheorem{theorem}{Theorem}
\newtheorem{corollary}{Corollary}

\bibliography{references}

\IEEEoverridecommandlockouts
\begin{document}

\title{\LARGE \bf
The Cost of Differential Privacy in Linear-Quadratic Dynamic Games
}

\author{
Chih-Yuan Chiu$^{1}$, Matthew Hale$^{1}$
\thanks{$^{1}$School of Electrical and Computer Engineering at the Georgia Institute of Technology, Atlanta, GA, USA.}
}

\maketitle

\thispagestyle{empty}
\pagestyle{empty}





\begin{abstract}
Multi-agent coordination often requires strategic agents to share sensitive information about their states or objectives, creating a tension between performance and privacy.
Our paper studies this tradeoff in stochastic linear-quadratic (LQ) dynamic games with heterogeneous agent objectives. 
In our framework, agents share noise-perturbed state and reference information with a cloud computer that computes feedback Nash equilibrium strategies, with the injected noise calibrated to provide differential privacy. 
We derive an analytical expression for each agent's infinite-horizon steady-state cost of privacy relative to the non-private game. 
Then, we prove that when agents' objectives are sufficiently aligned, the injection of privacy noise necessarily incurs a positive performance cost. 
In contrast, we characterize a class of games with sufficiently misaligned objectives across agents for which an agent's cost of privacy can be strictly negative. 
Thus, counterintuitively, noise can simultaneously protect privacy and improve the equilibrium performance of an agent when the objectives of interacting agents are sufficiently misaligned. 
Finally, we present numerical experiments which corroborate our theoretical contributions.
\end{abstract}

\section{Introduction}
\label{sec: Introduction}

Strategic multi-agent decision-making arises throughout cyber-physical and socio-technical systems, including communication, power, and transportation networks \cite{Ratliff2018APerspectiveonIncentiveDesignChallengesandOpportunities, Saad2009CoalitionalGameTheoryforCommunicationNetworks, Ndrio2023CoordinatedTransactionSchedulinginMultiAreaElectricityMarkets,
Roughgarden2002HowBadIsSelfishRouting, BaillonCominetti2008MarkovianTrafficEquilibrium}. 
In such systems, achieving desirable interaction outcomes often requires agents to coordinate by sharing information about their states or objectives. However, such information may be sensitive, creating a fundamental tension between the information required for effective coordination and each agent's desire to preserve its privacy \cite{McDaniel2009SecurityandPrivacyChallengesintheSmartGrid, Nissim2012PrivacyawareMechanismDesign}.

\textit{Differential privacy} provides a principled framework for navigating this tension by enabling agents to obscure sensitive information by injecting calibrated noise, while retaining its utility for downstream tasks \cite{Dwork2006DifferentialPrivacy,
Hawkins2026DifferentiallyPrivateFormationControl, 
NyPappas2014DifferentiallyPrivateFiltering,
HaleJonesLeahy2018PrivacyInFeedback,
YazdaniHale2020ErrorBoundsandGuidelinesForPrivacyCalibrationinDifferentiallyPrivateKalmanFiltering}.
Originally developed for protecting database entries, differential privacy is preserved under post-processing and robust to side information \cite{Dwork2006DifferentialPrivacy, Dwork2014AlgorithmicFoundationsOfDifferentialPrivacy, Yazdani2023DifferentiallyPrivateLQControl, Kasiviswanathan2008OnTheSematicsofDifferentialPrivacy}, and is thus well-suited to estimation, filtering, and control.
In multi-agent control contexts, for instance, agents may perturb the information they share for coordination \cite{HaleJonesLeahy2018PrivacyInFeedback, Yazdani2023DifferentiallyPrivateLQControl}.
The resulting privacy noise, however, reduces the accuracy of the information available to agents' controllers 
and thus typically degrades closed-loop system performance.

Existing work on differentially private control has largely focused on cooperative settings with shared estimation \cite{NyPappas2014DifferentiallyPrivateFiltering, YazdaniHale2020ErrorBoundsandGuidelinesForPrivacyCalibrationinDifferentiallyPrivateKalmanFiltering} or control objectives \cite{HaleJonesLeahy2018PrivacyInFeedback, Hawkins2026DifferentiallyPrivateFormationControl}.
Under heterogeneous objectives, however, privatization may benefit some agents while disadvantaging others.
It is therefore unclear whether privacy must impose a performance penalty on every agent in a noncooperative dynamic game.


To investigate this question, we utilize the lens of differential privacy to study the $N$-agent linear-quadratic Gaussian (LQG) dynamic game, a stochastic variant of the well-established linear-quadratic dynamic game framework for strategic interactions under heterogeneous agent objectives \cite{basar1998DynamicNoncooperativeGameTheory, LiuLi2025PolicieswithSparseInterAgentDependenciesinDynamicGames, 
FridovichKeil2020iLQGames}. 
In our LQG dynamic game, each agent's cost captures its preferences over state-reference tracking and control minimization objectives, with these preferences allowed to differ across agents. To facilitate coordination, agents communicate observations of their states and reference states to a cloud computer, which computes feedback Nash equilibrium strategies from the information it receives. 

Agents may inject privacy-preserving noise into the transmitted information, thereby trading privacy enforcement against the extent to which their communication is informative. 
We quantify the performance implications of this privacy mechanism through a control-theoretic \textit{cost of privacy} for each agent, defined as the difference between the agent's infinite-horizon steady-state costs with and without privacy noise.
Unlike in cooperative settings~\cite{chen23}, the cost of privacy 
is neither shared across agents nor, crucially, always positive.

Our contributions are threefold:
\begin{enumerate}
    \item We analytically quantify each agent's cost of privacy in our LQG game model, 
    generalizing
    existing cost-of-privacy analyses for cooperative linear-quadratic control to strategic settings with heterogeneous agent objectives. 
    \item We establish conditions under which sufficiently aligned agent objectives induce a positive cost of privacy, recovering analogous results expected in cooperative settings. 
    \item We characterize a class of LQG games with sufficiently misaligned objectives for which an agent's cost of privacy can instead be strictly negative. Thus, privacy noise can, in certain strategic interactions, simultaneously encoding privacy while improving cost performance.
\end{enumerate}
Numerical experiments corroborate our theoretical contributions in both the positive- and negative-cost regimes.


\subsection{Notation}
\label{subsec: Notation}

Given non-negative integers $N_1, N_2$, with $N_1 < N_2$, we write $[N_1, N_2] := \{N_1, \cdots, N_2\}$ and $[N_2] := [1, N_2]$.
Given $w_1, \cdots, w_T \in \R^n$, and $t_1, t_2 \in [0, T]$ with $t_1 \leq t_2$, we set $w_{t_1:t_2} := (w_{t_1}, \cdots, w_{t_2}) \in \R^{n(t_2 - t_1 + 1)}$.
Subscripts or superscripts omitted from indexed variables denote concatenation over the corresponding index; a superscript $-i$ denotes concatenation over all agents except $i$.
Let $\sym^n$, $\symPSD^n$, and $\symPD^n$ denote the sets of symmetric, symmetric positive semidefinite, and symmetric positive definite $n \times n$ matrices, respectively.
Given matrices $M_1, \cdots, M_n$,
we use $\diag\{M_1, \cdots, M_n\}$ to denote the block diagonal matrix whose constituent block matrices are $M_1, \cdots, M_n$, and, if $M_1, \cdots, M_n$ have the same column count, we use $\vstack(M_1, \cdots, M_n)$ to denote their vertical concatenation.
The multivariate Gaussian distribution with mean $\mu \in \R^n$ and covariance $\Sigma \in \R^{n \times n}$ 
is denoted by $\normalVec(\mu, \Sigma)$.
Given $n, m \in \N$, we respectively denote an $n \times n$ identity matrix and an $m \times n$ zero matrix by $I_{n \times n}$ and $O_{m \times n}$. 
We denote the Gaussian tail integral
by $\mathcal{Q}(z) = \frac{1}{\sqrt{2\pi}} \int_{z}^{\infty} \exp\big(-\frac{u^2}{2}\big) du$.

\section{Preliminaries}
\label{sec: Preliminaries}



\subsection{Dynamic LQG Games and Nash Equilibria Definitions}
\label{subsec: Dynamic LQG Games and Nash Equilibria Definitions}


Consider an $N$-agent, $T$-stage stochastic linear quadratic dynamic game $\G_T := (A, B, C, \boldQ, \boldR)$, as follows. For each $i \in [N] := \{1, \cdots, N\}$ and $t \in [0, T+1] := \{0, 1, \cdots, T+1\}$, let $x_t^i \in \R^{n_x^i}$, $u_t^i \in \R^{n_u^i}$, and $y_t^i \in \R^{n_y^i}$ respectively denote the state, control, and (noise corrupted) output of Agent $i$ at time $t$. 
For each $t \in [0, T]$, we 
write
$x_t := (x_t^1, \cdots, x_t^N) \in \R^{n_x}$ and $u_t := (u_t^1, \cdots, u_t^N) \in \R^{n_u}$, where $n_x := \sum_{i=1}^N n_x^i$ and $n_u := \sum_{i=1}^N n_u^i$.
We
also write
$x_{0:T+1} := (x_0, \cdots, x_T) \in \R^{n_x (T+2)}$ and $u_{0:T} := (u_0, \cdots, u_T) \in \R^{n_u (T+1)}$.

We consider stochastic time-invariant dynamics specified by a gain matrix $A \in \R^{n \times n}$ for the $N$-agent system, 
an actuation matrix $B^i \in \R^{n \times n_u^i}$ for each~$i \in [N]$ (with $B := \begin{bmatrix}
    B^1 & \cdots & B^N
\end{bmatrix} \in \R^{n \times m}$ for brevity), 
a stochastic initial state $x_0 \in \R^n$ distributed as $x_0 \sim \mathcal{N}(0, X_0)$ with $X_0 \in \symPD^n$, and system-level dynamics noise terms $w_t \sim \mathcal{N}(0, \Sigma_w)$, with $\Sigma_w \in \symPD^n$,
across $t \in [0, T]$. 
Then, $\forall t \in [0, T]$:
{
\setlength{\abovedisplayskip}{2pt}
\setlength{\belowdisplayskip}{3pt}
\begin{align} 
\label{Eqn: LQ Game, dynamics map}
    x_{t+1} &= A x_t + \sum_{i=1}^N B^i u_t^i + w_t = A x_t + B u_t + w_t.
\end{align}
}

Moreover, each agent $i \in [N]$ is associated with a cost $\tilde J_{0:T}^i(x_{0:T+1}, u_{0:T})$, given below in \eqref{Eqn: Cost, Agent i (tilde J superscript i)}, which measures the extent to which Agent $i$ values how closely each agent $j \in [N]$ tracks a \textit{reference trajectory} $\bar x_t^j$, as well as the control effort expended by each agent $j \in [N]$.
To facilitate coordination while preserving privacy, each Agent $i$ shares noisy representations of its state $x_t^i$ and reference $\bar{x}_t^i$ through a cloud computer \cite{Yazdani2023DifferentiallyPrivateLQControl}. 
Similar to \cite[Alg. 1]{Yazdani2023DifferentiallyPrivateLQControl}, the cloud in our work aggregates partial state information across agents, computes Nash equilibrium strategies (as defined and characterized below), and distributes the resulting control inputs to the agents. 
Such cloud architectures have been employed in cyber-physical systems for state monitoring and estimation \cite{Soudbakhsh2014DelayawareCodesignsforWideAreaControlofPowerGrids, Singh2016DecentralizedControlofOscillatoryDynamicsinPowerSystemsUsinganExtendedLQR, Yazdani2023DifferentiallyPrivateLQControl}.
In our setting, the cloud acts as a \say{mediator} that coordinates the $N$ agents, but is not trusted with their sensitive information.

Concretely, at each time $t$, each agent $i$ broadcasts:
{
\setlength{\abovedisplayskip}{3pt}
\setlength{\belowdisplayskip}{3pt}
\begin{alignat}{2} 
\label{Eqn: LQ Game, output map}
    y_t^i &= C^i x_t^i + v_t^i, \quad &&\forall \ i \in [N], \ t \in [0, T+1], \\
\label{Eqn: LQ Game, noisy reference}
    \tilde x_t^i &= \bar x_t^i + \bar w_t^i, &&\forall \ i \in [N], \ t \in [0, T+1].
\end{alignat}
}
$\hspace{-1mm}$to the cloud computer, where $C^i \in \R^{n_y^i \times n_x^i}$, $\Sigma_w^i \in \symPD^{n_x^i}$, and $\Sigma_v^i \in \symPD^{n_y^i}$, and $v_t^i\sim\mathcal N(0,\Sigma_v^i)$ and $\bar w_t^i\sim\mathcal N(0,\Sigma_{\bar w}^i)$ encode privacy over its state and reference trajectories, respectively.
The notation $\tilde x_t^i$ and $y_t^i$ can be aggregated across time and agent variables in the manner described in Sec. \ref{subsec: Notation}.
Similarly, we write $\Sigma_v := \diag\{\Sigma_v^1, \cdots, \Sigma_v^N \} \in \symPD^{n_y}$ and $\Sigma_{\bar w} := \diag\{\Sigma_{\bar w}^1, \cdots, \Sigma_{\bar w}^N \} \in \symPD^{n_x}$.

For each agent $i \in [N]$, to define the costs $\tilde J_{0:T}^i$, we specify $Q^i \in \sym^n$ and $R^{ij} \in \sym^{m_j}$.
We assume that $R^{ii} \in \symPD^{n_x^i}$ for each $i \in [N]$, but do not restrict either $Q^i$, or any $R^{ij}$ with $j \in [N] \backslash \{i\}$, to be positive semidefinite or positive definite.
For convenience, we introduce the shorthand notations $z_t := (w_{0:t}, v_{0:t+1}, \bar w_{0:T+1})$ for each $t \in [0, T]$.
We then define:
{
\setlength{\abovedisplayskip}{3pt}
\setlength{\belowdisplayskip}{3pt}
\begin{align}
\label{Eqn: Cost, Agent i (tilde J superscript i)}
    &\tilde J_{0:T}^i(x_{0:T+1}, u_{0:T}) \\
\nonumber
    := \ &\E_{z_T}\big[ \textstyle\sum_{t=0}^T (x_t - \tilde x_t)^\top Q^i (x_t - \tilde x_t) + \textstyle\sum_{j = 1}^N u_t^{j\top} R^{ij} u_t^j \\ \nonumber
    &\hspace{1.0cm} + (x_{T+1} - \tilde x_{T+1})^\top Q^i (x_{T+1} - \tilde x_{T+1}) \big].
\end{align}
}
We write $R^i := \diag\{R^{i1}, \cdots, R^{iN}\} \in \R^{m \times m}$, and use $\boldQ := \{Q^i, \cdots, Q^N \}$ and 
$\boldR := \{R^i, \allowbreak \cdots, R^N \}$
to respectively characterize the sets of state and control penalty matrices in the above definition of $\tilde J_{0:T}^i$.

We
consider dynamic games under \textit{closed-loop imperfect state} (CLIS) information \cite[Def. 5.2(iii)]{basar1998DynamicNoncooperativeGameTheory}, in which the cloud computer receives the information $\info_t := (y_{0:t}, \tilde x_{0:T+1})$ at each time $t \in [0, T+1]$, comprising the observation stream $\{y_0^i, \cdots, y_t^i \}$ across all agents $i \in [N]$, as well as the noise-corrupted reference trajectory $\tilde x_{0:T+1}$.
The cloud then designs \textit{output feedback strategies} $\gamma_t^i: \R^{n_y^i(t+1)} \times \R^{n_x(T+1)} \ra \R^{n_u^i}$ for each agent $i \in [N]$ and time $t \in [0, T]$, which constitute a \textit{feedback Nash equilibrium} (FBNE) under the CLIS information pattern, as described below.

\begin{definition}
(\textbf{Feedback Nash Equilibrium (FBNE) for Finite-Horizon LQG Games})
\label{Def: FBNE under CLIS}
We call a given set of output feedback strategies $\boldgamma^\star := \{\gamma_t^{i\star}: \R^{n_y^i(t+1)} \times \R^{n_x(T+1)} \ra \R^{n_u^i}: i \in [N], t \in [0, T] \}$ a FBNE for the game $\G_T = (A, B, C, \boldQ, \boldR)$ with respect to the CLIS information pattern if its cost-to-go maps $W_t^i: \R^{n_y^i(t+1)} \times \R^{n_x(T+1)} \times \R^m \ra \R$ and minimum cost-to-go maps $V_t^i: \R^{n_y^i(t+1)} \times \R^{n(T+1)} \ra \R$, defined recursively via:
{
\setlength{\abovedisplayskip}{3pt}
\setlength{\belowdisplayskip}{3pt}
\small
\begin{subequations}
\begin{align}
\label{Eqn: Cost-to-go (W) initialization, at time T+1}
    W_{T+1}^i(&\info_{T+1}, u_{T+1})  \\
    \nonumber
    &:= \E_{z_T} \big[ (x_{T+1} - \tilde x_{T+1})^\top Q^i (x_{T+1} - \tilde x_{T+1}) | \info_{T+1} \big], \\
\label{Eqn: Optimal Cost-to-go (V), Def}
    V_t^i(\info_t) &:= W_t^i\big( \info_t, \gamma_t^\star(\info_t) \big), \quad \forall \ t = T+1, \cdots, 0, \\
\label{Eqn: Cost-to-go (W) recursive DP update}
    W_t^i(\info_t, u_t) &:= \E_{z_t}\big[ (x_t - \bar x_t)^\top Q^i (x_t - \bar x_t) \\
    \nonumber
    &\hspace{1cm} + u_t^\top R^i u_t + V_{t+1}^i(\info_{t+1}) \big| \info_t \big], \\ \nonumber
    &\hspace{2.5cm} \forall \ t = T+1, \cdots, 0,
\end{align}
\end{subequations}
}
$\hspace{-1.4mm}$satisfy the unilateral optimality condition for each agent, i.e., for each $i \in [N]$, $t \in [0, T]$, and alternative choice of output feedback map $\gamma_t^i: \R^{n_y^i(t+1)} \times \R^{n(T+1)} \ra \R^n$:
{
\setlength{\abovedisplayskip}{3pt}
\setlength{\belowdisplayskip}{3pt}
\begin{align}
\nonumber
    V_t^i(\info_t) &:= W_t^i\big( \info_t, \gamma_t^{i\star}(\info_t), \gamma_t^{-i\star}(\info_t) \big) \\
\nonumber
    &\leq W_t^i\big( \info_t, \gamma_t^i(\info_t), \gamma_t^{-i\star}(\info_t) \big).
\end{align}
}
\end{definition}

\subsection{Extension to Infinite-Horizon Dynamic LQG Games}
\label{subsec: Extension to Infinite-Horizon Dynamic LQG Games}

An infinite-horizon variant $G_\infty := (A, B, C, \boldQ, \boldR)$ of the dynamic LQG game and Nash equilibrium concept introduced in Sec. \ref{subsec: Dynamic LQG Games and Nash Equilibria Definitions} can be derived as follows.
We make the mild assumption, standard in the infinite-horizon optimal control literature \cite{Bertsekas1995DynamicProgrammingAndOptimalControl, Yazdani2023DifferentiallyPrivateLQControl}, that 
$\lim_{t \ra \infty} \bar x_t^i = \bar x^i$, and thus the agents' reference trajectories may be assumed static in the infinite-horizon setting.

We consider infinite-horizon output feedback strategies $\boldgamma := \{ \gamma_t^i: \R^{n_y^i(t+1)} \times \R^n \ra \R^{n_u^i}: i \in [N], t \in \{0 \} \cup \N \}$, where, for each $i \in [N]$ and $t \in \{0\} \cup \N$, each output feedback strategy $\gamma_t^i: (\info_t, \bar x) \mapsto u_t^i$ maps the information available to Agent $i$ at time $t$, i.e., $\info_t := (y_0, \cdots, y_t) \in \R^{n_y(t+1)}$, and the agents' (limiting) reference state $\bar x \in \R^n$, to their control at time $t$, i.e., $u_t^i$.
For convenience, define $x_{0:\infty} := (x_0, x_1, \cdots)$ and $u_{0:\infty} := (u_0, u_1, \cdots)$ to be the countably infinite sequences of states and controls associated with 
$G_\infty$.
Then the infinite-horizon extension of each Agent $i$'s finite-horizon cost \eqref{Eqn: Cost, Agent i (tilde J superscript i)} can be formulated as follows:
{
\setlength{\abovedisplayskip}{3pt}
\setlength{\belowdisplayskip}{3pt}
\begin{align}
\nonumber
    \tilde J^i(x_0, \boldgamma) &:= \lim_{T \ra \infty} \frac{1}{T+1} \E_{z_T}\big[ \textstyle\sum_{t=0}^T (x_t - \tilde x_t)^\top Q^i (x_t - \tilde x_t) \\
\label{Eqn: Cost, Agent i (tilde J superscript i), Infinite-horizon}
    &\hspace{3.2cm} + \textstyle\sum_{j = 1}^N u_t^{j\top} R^{ij} u_t^j \big]. 
\end{align}
}

We then define the infinite-horizon \textit{feedback Nash equilibrium} (FBNE) under the CLIS information pattern as follows.

\begin{definition}
(\textbf{Feedback Nash Equilibrium (FBNE) for Infinite-Horizon LQG Games})
\label{Def: FBNE under CLIS, Infinite-horizon}
We call a given set of output feedback strategies $\boldgamma^\star := \{\gamma_t^{i\star}: \R^{n_y^i(t+1)} \times \R^{n_x} \ra \R^{n_u^i}: i \in [N], t \in \{0\} \cup \N \}$ a feedback Nash Equilibrium with respect to the CLIS information pattern for the infinite-horizon dynamic game $\G_\infty$ if, for any agent $i \in [N]$, and any alternative choice of output feedback strategies $\boldgamma_\steady^i := \{\gamma_t^i: \R^{n_y^i(t+1)} \times \R^{n_x} \ra \R^{n_u^i}: t \in \{0\} \cup \N \}$ for Agent $i$:
{
\setlength{\abovedisplayskip}{3pt}
\setlength{\belowdisplayskip}{3pt}
\begin{align}
    \tilde J^i(x_0, \boldgamma_\steady^{i\star}, \boldgamma_\steady^{-i\star}) &\leq \tilde J^i(x_0, \boldgamma_\steady^i, \boldgamma_\steady^{-i\star}).
\end{align}
}
\end{definition}

\subsection{Background on Differential Privacy}
\label{subsec: Fundamentals of Differential Privacy}


In our model, 
the dynamics noise terms $w_t$ represent exogenous disturbances which arise independently of agents' decisions, while the distributions of the output noise terms $v_t^i$ and reference noise terms $\bar w_t^i$ are designed by agents to preserve the privacy of their state and reference trajectories, respectively. 
Below, we describe guiding principles, grounded in
the theory of 
differential privacy, for agents to select appropriate output and reference noise levels to enforce differential privacy.

Differential privacy provides a mathematical framework for protecting sensitive data, while allowing post-hoc accurate analyses, filtering, and processing that do not weaken its privacy protections.
Our work considers an input perturbation-based form of differential privacy.
In this formulation, agents add privacy-protecting noise to sensitive data before it is shared, thus ensuring that 
it is difficult for the recipient to determine the underlying sensitive data that produced it. 
All background on differential privacy is presented for the input perturbation case, and
we refer the reader to~\cite{NyPappas2014DifferentiallyPrivateFiltering} for a more general background. 

Concretely, for each agent $i \in [N]$, we consider finite-horizon\footnote{The 
discussion in Sec. \ref{subsec: Fundamentals of Differential Privacy} on input perturbation-based differential privacy extends naturally to infinite-horizon settings; see, e.g., \cite[Sec. 2]{Yazdani2023DifferentiallyPrivateLQControl}.} 
state trajectories $x^i := (x_0^i, \cdots, x_{T+1}^i) \in \R^{n_x^i (T+2)}$
and reference trajectories~$\bar{x}^i = (\bar{x}_0^i, \cdots, \bar{x}_{T+1}^i) \in \R^{n_x^i (T+2)}$. 
The goal of differential privacy is to make such trajectories approximately indistinguishable from
other nearby trajectories, and the notion
of ``nearby'' is codified by an adjacency relation.
For simplicity, we state the adjacency relation and the definition of differential privacy in terms of~$x^i$,
though both definitions apply verbatim to~$\bar{x}^i$ as well. 

\begin{definition} \label{def:adj}
\looseness=-1
Fix an adjacency parameter~$b^i > 0$ for agent~$i \in [N]$. Two trajectories~$x^i = \big(x^i_0, \cdots, x^i_T), \hat x^i = (\hat x^i_0, \cdots, \hat x^i_T) \in \R^{n_x^i T}$
are adjacent if~$\|x^i - \hat x^i\|_2 \leq b^i$. 
\end{definition}

The adjacency relation in Def.~\ref{def:adj} requires that a 
trajectory be made approximately indistinguishable from all other trajectories not
more than 
$b^i$ away
in the~$2$-norm.
Approximate indistinguishability is enforced by a privacy mechanism
(a randomized map that implements differential privacy). 

\begin{definition} \label{def:dp}
Let~$\epsilon^i \geq 0$ and~$\delta^i \in (0, \frac{1}{2})$ be given. A mechanism~$M$ 
is~$(\epsilon^i, \delta^i)$-differentially private if, for all adjacent~$x^i, \hat x^i \in \R^{n_x^i T}$:
{
\setlength{\abovedisplayskip}{3pt}
\setlength{\belowdisplayskip}{3pt}
\begin{align}
\nonumber
    \mathbb{P}[M(x^i) \in S] \leq e^{\epsilon^i}\mathbb{P}[M(\hat x^i) \in S] + \delta^i
\end{align}
}
$\hspace{-1mm}$for all sets~$S$ in the Borel~$\sigma$-algebra over~$\R^{n_x^i T}$. 
\end{definition}

Above, $\epsilon^i$ quantifies the leakage of sensitive information and typically ranges from $0.1$ to $10$ \cite{hsu14}, while $\delta^i$ is the probability of leaking more information than $\epsilon^i$-leakage allows, and typically ranges from $0$ to $0.05$ \cite{hsu14}.
Our work enforces differential privacy by applying the Gaussian mechanism to 
the information that the agents send to the cloud in~\eqref{Eqn: LQ Game, output map}-\eqref{Eqn: LQ Game, noisy reference}. 

\begin{proposition} \label{Prop: Calibrating Privacy and Adjacency Parameters}
Suppose Agent~$i$ uses privacy parameters~$\epsilon^i \geq 0$ and~$\delta^i \in (0, \frac{1}{2})$,
with adjacency parameter~$b^i > 0$. 
Then the output map
in
\eqref{Eqn: LQ Game, output map}
is~$(\epsilon^i, \delta^i)$-differentially private if~$v^i_t \sim \mathcal{N}(0, (\sigma^i)^2 I_{n^i_y \times n^i_y})$,
where 
$\sigma^i \geq \frac{\Vert  C^i \Vert_2 b^i}{2\epsilon^i} \big(\mathcal{Q}^{-1}(\delta^i) + \sqrt{(\mathcal{Q}^{-1}(\delta^i))^2 + 2\epsilon^i} \big)$. 
The map~$\bar{x}^i_t \mapsto \tilde{x}^i_t = \bar{x}^i_t + \bar{w}^i_t$ is~$(\epsilon^i, \delta^i)$-differentially private
if~$\bar{w}^i_t \sim \mathcal{N}(0, (\tau^i)^2 I_{n^i_x \times n^i_x})$, where
$\tau^i \geq \frac{b^i}{2\epsilon^i} \big(\mathcal{Q}^{-1}(\delta^i) + \sqrt{(\mathcal{Q}^{-1}(\delta^i))^2 + 2\epsilon^i}\big)$. 
\end{proposition}

For the rest of the paper, we suppose that each noise term~$v^i_t$ 
and~$\bar{w}^i_t$ 
satisfies its variance bound in Prop. \ref{Prop: Calibrating Privacy and Adjacency Parameters} for all agents~$i \in [N]$ and all times~$t \in [0, T+1]$.

\section{Differentially Private Nash Equilibrium Controls in LQG Games}
\label{sec: Differentially Private Nash Equilibrium Controls in LQG Games}


\subsection{Finite-Horizon Equilibrium Strategies for LQG Games}
\label{subsec: Finite-Horizon Equilibrium Strategies for LQG Games}


We now present Riccati equations for computing 
the FBNE $\boldgamma^\star$ defined in Def. \ref{Def: FBNE under CLIS} under a set of mild regularity conditions (Assumption \ref{Assump: Phi and Z matrices are invertible}).
To facilitate the characterization of the FBNE, we first present the following notation. 
Given an $N$-agent, horizon-$T$ dynamic game $\G_T := (A, B, C, \boldQ, \boldR)$,
we define $H_t \in \R^{m \times n}$, $\Phi_t \in \R^{m \times m}$, $L_t \in \R^{m \times nN}$, $\tilde h_t \in \R^m$, $\tilde g_t \in \R^m$, $Z_t^i \in \R^{n \times n}$, $\tilde \beta_t^i \in \R^n$, and $\tilde c_t^i \in \R$ for each $i \in [N]$ via backward recursion from $t = T+1$, as follows:
{
\setlength{\abovedisplayskip}{3pt}
\setlength{\belowdisplayskip}{3pt}
\begin{subequations}
\label{Eqn: Riccati equations for FBNE under CLIS}
\begin{align}
\label{Eqn: Cost-to-go matrix for Dynamic LQ Games (Z) initialization, at time T+1}
    Z_{T+1}^i &:= Q^i, \\
\label{Eqn: Cost-to-go vector for Dynamic LQ Games (beta) initialization, at time T+1}
    \tilde \beta_{T+1}^i &:= \tilde x_{T+1}, \\
\label{Eqn: Cost-to-go bias term for Dynamic LQ Games (c) initialization, at time T+1}
    \tilde c_{T+1}^i &= \tr(Q^i \Sigma_{T+1|T+1}), \\
\label{Eqn: H t}
    H_t &:= \vstack(B^{1\top} Z_{t+1}^1, \cdots, B^{N\top} Z_{t+1}^N), \\
\label{Eqn: Phi t}
    \Phi_t &:= \diag\{R^{11}, \cdots, R^{NN} \} + H_t B,  \\
\label{Eqn: L t}
    L_t &:= - \Phi_t^{-1} H_t A,  \\
\label{Eqn: tilde h t}
    \tilde h_t &:= \vstack(B^{1\top} Z_{t+1}^1 \tilde \beta_{t+1}^1, \cdots, B^{N\top} Z_{t+1}^N \tilde \beta_{t+1}^N) \\
\label{Eqn: tilde g t}
    \tilde g_t &:= \Phi_t^{-1} \tilde h_t, \\
\label{Eqn: Z t i}
    Z_t^i &:= Q^i + L_t^\top R^i L_t + (A + B L_t)^\top Z_{t+1}^i (A + B L_t), \\ 
\label{Eqn: beta t i}
    \tilde \beta_t^i &:= (Z_t^i)^{-1} \big( Q^i \tilde x_t - L_t^\top R^i \tilde g_t \\
    \nonumber
    &\hspace{1cm} - (A + B L_t)^\top Z_{t+1}^i (B \tilde g_t - \tilde \beta_{t+1}^i) \big), \\ 
\nonumber
    \tilde c_t^i &:= \tilde c_{t+1}^i + \tr(Q^i \Sigma_{t|t}) + \tr(Z_{t+1}^i (\Sigma_{t+1|t} - \Sigma_{t+1|t+1})) \\
    \nonumber
    &\hspace{1cm} + \tilde x_t^\top Q^i \tilde x_t + \textstyle\sum_{j \in [N]} \tilde g_t^{j\top} R^{ij} \tilde g_t^j - \tilde \beta_t^{i\top} Z_t^i \tilde \beta_t^i  \\
    \label{Eqn: c t i}
    &\hspace{1cm} + (B \tilde g_t - \tilde \beta_{t+1}^i)^\top Z_{t+1}^i (B \tilde g_t - \tilde \beta_{t+1}^i),
\end{align}
\end{subequations}
}

To characterize the \textit{maximum a posteriori} (MAP) state estimates for each agent computed by the cloud across time, we use the following Kalman filter update equations (see, e.g., \cite{Bertsekas1995DynamicProgrammingAndOptimalControl}, for a detailed derivation).
For each $t', t \in [0, T+1]$ such that $t' \geq t$, define:\footnote{For Gaussian variables, the MAP estimate and conditional mean coincide.}
{
\setlength{\abovedisplayskip}{3pt}
\setlength{\belowdisplayskip}{3pt}
\begin{subequations}
\begin{align}
    \hat x_{t'|t}(\info_t) &:= \E[x_{t'}|\info_t], \\
    \Sigma_{t'|t}(\info_t) &:= \cov(x_{t'}|\info_t);
\end{align}
\end{subequations}
}
$\hspace{-1mm}$Next, for each $t \in [0, T]$, we define:
{
\setlength{\abovedisplayskip}{3pt}
\setlength{\belowdisplayskip}{3pt}
\begin{align}
    M_{t+1} &:= \Sigma_{t+1|t} C^\top (C \Sigma_{t+1|t} C^\top + \Sigma_v)^{-1}, \\
    e_{t+1} &:= y_{t+1} - C \hat x_{t+1|t}.
\end{align}
}
$\hspace{-1mm}$Then, across $t \in [0, T]$, the Kalman filter updates are:
{
\setlength{\abovedisplayskip}{3pt}
\setlength{\belowdisplayskip}{3pt}
\begin{subequations}
\label{Eqn: Kalman filter updates}
\begin{align}
\label{Eqn: Kalman filter, mean, dynamics update}
    \hat x_{t+1|t} &= A \hat x_{t|t} 
    + B u_t, \\
\label{Eqn: Kalman filter, covariance, dynamics update}
    \Sigma_{t+1|t} &= A \Sigma_{t|t} A^\top + \Sigma_w, \\
\label{Eqn: Kalman filter, mean, output update}
    \hat x_{t+1|t+1} 
    &= A \hat x_{t|t} + B u_t + M_{t+1} e_{t+1}, \\
\label{Eqn: Kalman filter, covariance, output update}
    \Sigma_{t+1|t+1} &= (\Sigma_{t+1|t}^{-1} + C^\top \Sigma_v^{-1} C)^{-1}.
\end{align}
\end{subequations}
}

Prop. \ref{Prop: FBNE under CLIS (LQ)} below extends \cite[Cor. 6.1]{basar1998DynamicNoncooperativeGameTheory} to establish that, when \eqref{Eqn: Riccati equations for FBNE under CLIS} is well-defined for each $t \in [0, T+1]$, the matrices $L_t$ and vectors $\tilde g_t$ (see \eqref{Eqn: L t} and \eqref{Eqn: tilde g t}) characterize the Nash equilibria $\boldgamma^\star$ for the game $\G_T$ under the CLIS information pattern.

We begin by presenting the following mild assumptions, standard in the control theory and dynamic games literature, which assert that the $N$-agent system \eqref{Eqn: LQ Game, dynamics map}-\eqref{Eqn: LQ Game, output map} is observable (Assumption \ref{Assump: System Observability}) and that the matrices involved in \eqref{Eqn: Riccati equations for FBNE under CLIS} are well-conditioned (Assumption \ref{Assump: Phi and Z matrices are invertible}).

\begin{assumption}
\label{Assump: System Observability}
The pair $(A, C)$ is observable.
\end{assumption}

\begin{assumption}
\label{Assump: Phi and Z matrices are invertible}
For all $i \in [N]$ and $t \leq T+1$, the matrices $\Phi_t$ and $Z_t^i$ (see \eqref{Eqn: Phi t} and \eqref{Eqn: Z t i}) are invertible.
\end{assumption}

\begin{proposition}(\textbf{Feedback Nash Equilibrium for LQ Games under CLIS})
\label{Prop: FBNE under CLIS (LQ)}
Under Assumption \ref{Assump: Phi and Z matrices are invertible}, the FBNE $\boldgamma := \{ \gamma_t^i \}$ for the LQ game $\G_T = (A, B, C, \boldQ, \boldR)$ is given by:
{
\setlength{\abovedisplayskip}{3pt}
\setlength{\belowdisplayskip}{3pt}
\begin{align}
\label{Eqn: FBNE under CLIS (LQ)}
    \gamma_t^i(\info_t) &= L_t^i \hat x_{t|t} + \tilde g_t^i, \qquad \forall \ i \in [N], t \in [0, T],
\end{align}
}
$\hspace{-1mm}$where $L_t^i \in \R^{m_i \times nN}$ and $\tilde g_t^i \in \R^{m_i}$ are respectively given by the $i$-th row blocks (i.e., rows $1 + \sum_{j=1}^{i-1} m_j$ to $\sum_{j=1}^i m_j$, inclusive) of $L_t$ and $\tilde g_t$.
\end{proposition}


\begin{proof}
See App. \ref{subsec: App, Proof of Prop: FBNE under CLIS (LQ)}.
\end{proof}

\subsection{Infinite-Horizon Equilibrium Strategies for LQG Games}
\label{subsec: Infinite-Horizon Equilibrium Strategies for LQG Games}

To quantify the steady-state cost of privacy for LQ games, we first derive the infinite-horizon feedback Nash equilibria for LQ games under the CLIS information pattern, generalizing Prop. \ref{Prop: FBNE under CLIS (LQ)} to the steady-state infinite-horizon setting. 
We shall require the infinite-horizon analog of the equations \eqref{Eqn: Riccati equations for FBNE under CLIS}, as given below in \eqref{Eqn: Riccati equations for FBNE under CLIS, Infinite-Horizon}, which establish equality constraints that characterize the matrices $H \in \R^{m \times Nn}$, $Z, \Pi \in \R^{Nn \times Nn}$, $\Phi \in \R^{m \times m}$, $L \in \R^{m \times nN}$ and vectors $\tilde h \in \R^m$, $\tilde g \in \R^m$, $\tilde \beta \in \R^n$ (provided all inverse matrices appearing below in \eqref{Eqn: Riccati equations for FBNE under CLIS, Infinite-Horizon} exist).
{
\setlength{\abovedisplayskip}{3pt}
\setlength{\belowdisplayskip}{3pt}
\begin{subequations}
\label{Eqn: Riccati equations for FBNE under CLIS, Infinite-Horizon}
\begin{align}
\label{Eqn: H, Infinite-Horizon}
    H &= \vstack(B^{1\top} Z^1, \cdots, B^{N\top} Z^N), \\
\label{Eqn: Phi, Infinite Horizon}
    \Phi &= \diag\{R^{11}, \cdots, R^{NN} \} + H B, \\
\label{Eqn: Pi, Infinite Horizon}
    \Pi &= \Phi^{-1}, \\
\label{Eqn: L, Infinite-Horizon}
    L &= - \Phi^{-1} H A = - \Pi HA, \\
\label{Eqn: tilde h, Infinite-horizon}
    \tilde h &= \vstack(B^{1\top} Z^1 \tilde \beta^1, \cdots, B^{N\top} Z^N \tilde \beta^N) \\
\label{Eqn: tilde g, Infinite-horizon}
    \tilde g &= \Phi^{-1} \tilde h, \\
\label{Eqn: Zi, Infinite-horizon}
    Z^i &= Q^i + L^\top R^i L + (A + B L)^\top Z^i (A + B L), \\
\label{Eqn: beta i, Infinite-horizon}
    \tilde \beta^i &= (Z^i)^{-1} \big( Q^i \tilde x - L^\top R^i \tilde g - (A + B L)^\top Z^i (B \tilde g - \tilde \beta^i) \big). 
\end{align}
\end{subequations}
}





Next, we formally assume the  existence of matrices $H$, $Z$, $\Phi$, $L$ and vectors $\tilde h$, $\tilde g$, $\tilde \beta$, satisfying \eqref{Eqn: Riccati equations for FBNE under CLIS, Infinite-Horizon}.
This is again a mild assumption, since computing infinite-horizon FBNE for linear quadratic games by taking the limit $t \rightarrow \infty$ of \eqref{Eqn: Riccati equations for FBNE under CLIS} is standard in the dynamic games literature, and yields \eqref{Eqn: Riccati equations for FBNE under CLIS, Infinite-Horizon}.

\begin{assumption}
\label{Assump: Phi and Z matrices are invertible, Infinite-horizon}
There exist matrices $H$, $Z^i$, $\Phi$, $L$ and vectors $\tilde h$, $\tilde g$, $\tilde \beta^i$ across all $i \in [N]$, with $\Phi$ and $Z^i$ invertible for all $i \in [N]$, that satisfy \eqref{Eqn: Riccati equations for FBNE under CLIS, Infinite-Horizon} and are respective limit points for $H_t$, $\Phi_t$, $L_t$, $\tilde h_t$, $\tilde g_t$, $Z_t^i$, and $\tilde \beta_t^i$, defined recursively 
via \eqref{Eqn: Riccati equations for FBNE under CLIS}, as $t \ra - \infty$. Moreover, $I - (A + BL)^\top$ is invertible, as is:
{
\setlength{\abovedisplayskip}{3pt}
\setlength{\belowdisplayskip}{3pt}
\footnotesize
\begin{align}
\label{Eqn: Invertibility of I + hat Pi hat Z inv hat L, Condition}
    I + \Phi^{-1} 
    \begin{bmatrix}
        B^1 \big(I - (A + BL)^\top \big)^{-1} \big[ L^\top R^1 + (A + BL)^\top Z^1 B \big] \\
        \vdots \\
        B^N \big(I - (A + BL)^\top \big)^{-1} \big[ L^\top R^N + (A + BL)^\top Z^N B \big]
    \end{bmatrix}.
\end{align}
}
\end{assumption}

We also note that, since the system \eqref{Eqn: LQ Game, dynamics map}-\eqref{Eqn: LQ Game, output map} is assumed to be observable (Assumption \ref{Assump: System Observability}), the covariance matrix updates given by the Kalman filter equations \eqref{Eqn: Kalman filter, covariance, dynamics update} and \eqref{Eqn: Kalman filter, covariance, output update} converge to a time-invariant Kalman filter as $t \ra \infty$ \cite{Bertsekas1995DynamicProgrammingAndOptimalControl, Yazdani2023DifferentiallyPrivateLQControl}. Concretely, 
at steady state,
we have $\Sigma_{t+1|t} \ra \Sigma$ and $\Sigma_{t|t} \ra \bar \Sigma$, where $\Sigma, \bar \Sigma \in \R^{n \times n}$ are given by:
{
\setlength{\abovedisplayskip}{3pt}
\setlength{\belowdisplayskip}{3pt}
\begin{subequations}
\label{Eqn: Kalman filter equations, Infinite-horizon}
\begin{align}
\label{Eqn: Kalman filter equations, Infinite-horizon, dynamics update (Sigma)}
    \Sigma &= 
    A \bar \Sigma A^\top + \Sigma_w, \\
\label{Eqn: Kalman filter equations, Infinite-horizon, output update (bar Sigma)}
    \bar \Sigma &= (\Sigma^{-1} + C^\top \Sigma_v^{-1} C)^{-1}.
\end{align}
\end{subequations}
}

Prop. \ref{Prop: FBNE under CLIS (LQ)} and \cite[Sec. 6.2.3]{basar1998DynamicNoncooperativeGameTheory} imply that, under Assumptions \ref{Assump: System Observability}, \ref{Assump: Phi and Z matrices are invertible}, and \ref{Assump: Phi and Z matrices are invertible, Infinite-horizon}, the infinite-horizon FBNE strategy across agents is given by the maps $\boldgamma_\steady = \{\gamma_\steady^i: x \mapsto L^i x + \tilde g^i, \forall i \in [N] \}$.


\section{Cost of Privacy}
\label{sec: Cost of Privacy}


\subsection{Cost of Privacy in a dynamic LQ Game}
\label{subsec: Cost of Privacy in a dynamic LQ Game}

We can now
characterize the steady-state cost of privacy for each agent, as follows. First, we require the following preliminary notation. Under Assumption \ref{Assump: Phi and Z matrices are invertible, Infinite-horizon}, we define:
{
\setlength{\abovedisplayskip}{3pt}
\setlength{\belowdisplayskip}{3pt}
\small
\begin{subequations}
\label{Eqn: Notation for Cost of Privacy}
\begin{align}
\label{Eqn: hat L}
    \hat L &:= \begin{bmatrix}
        L^\top R^1 + (A + BL)^\top Z^1 B \\
        \vdots \\
        L^\top R^N + (A + BL)^\top Z^N B
    \end{bmatrix}, \\
\label{Eqn: hat Z}
    \hat Z &:= \diag\big\{ (I - (A + BL)^\top) Z^1, \cdots, (I - (A + BL)^\top) Z^N \big\}, \\
\label{Eqn: hat Q}
    \hat Q &:= \vstack\big( Q^1, \cdots, Q^N \big), \\
\label{Eqn: hat Pi}
    \hat \Pi &:= \begin{bmatrix}
        \Pi^{11} B^{1\top} Z^1 & \cdots & \Pi^{1N} B^{N\top} Z^N \\
        \vdots & \ddots & \vdots \\
        \Pi^{N1} B^{1\top} Z^1 & \cdots & \Pi^{NN} B^{N\top} Z^N 
    \end{bmatrix} \\
\label{Eqn: F}
    F &:= (I + \hat \Pi \hat Z^{-1} \hat L)^{-1} \hat \Pi \hat Z^{-1} \hat Q,
\end{align}
\end{subequations}
}
$\hspace{-1.5mm}$where $\Pi^{ij} \in \R^{m_i \times m_j}$ denotes the $(i, j)$-th block of $\Pi$, for each $i, j \in [N]$.
Note that the invertibility of $I + \hat \Pi \hat Z^{-1} \hat L$ can be derived from the invertibility of \eqref{Eqn: Invertibility of I + hat Pi hat Z inv hat L, Condition} in Assumption \ref{Assump: Phi and Z matrices are invertible, Infinite-horizon} via basic matrix operations, and thus $F$ is well-defined in \eqref{Eqn: F}.

\begin{theorem}
\label{Thm: Cost of Privacy}
Suppose Assumptions \ref{Assump: System Observability}, \ref{Assump: Phi and Z matrices are invertible}, and \ref{Assump: Phi and Z matrices are invertible, Infinite-horizon} hold.
For each $i \in [N]$, let $J^i(x_0, \boldgamma_\steady)$ denote the cost $\tilde J^i(x_0, \boldgamma_\steady)$ of Agent $i$ when no agent privatizes outputs or references, i.e., when $\Sigma_{\bar w} = O$ and $\Sigma_v = O$. 
Let $\Delta J^i(x_0, \boldgamma_\steady) := \tilde J^i(x_0, \boldgamma_\steady) - J^i(x_0, \boldgamma_\steady)$ denote the cost of privacy of Agent $i$.
Then:
{
\setlength{\abovedisplayskip}{3pt}
\setlength{\belowdisplayskip}{3pt}
\begin{align}
\nonumber
    \Delta J^i(x_0, \boldgamma_\steady) &= \tr\big(Z^i \Sigma + (Q^i - Z^i) \bar \Sigma \big) - \tr(Z^i \Sigma_w) \\
\label{Eqn: Cost of Privacy}
    &\hspace{1cm} + \tr \big( (Q^i + F^\top R^i F) \Sigma_{\bar w} \big).
\end{align}
}
\end{theorem}

\begin{proof}
See App. \ref{subsec: App, Proof of Thm. Cost of Privacy}.
\end{proof}

From Thm. \ref{Thm: Cost of Privacy}, we can readily establish that when each agent's objective function $J^i$ is convex, each agent’s cost of privacy is non-negative.

\begin{corollary}
\label{Cor: Positive Cost of Privacy}
Suppose Assumptions \ref{Assump: System Observability}, \ref{Assump: Phi and Z matrices are invertible}, and \ref{Assump: Phi and Z matrices are invertible, Infinite-horizon} hold, and that agents may privatize references but not outputs, i.e., $\Sigma_v^j = 0, \ \ \forall \ j \in [N]$.
If $Q^i \in \symPSD^n$ and $R^i \in \symPSD^m$ for each agent $i \in [N]$, then $\Delta J^i(x_0, \boldgamma_\steady) \geq 0$ for each $i \in [N]$.
\end{corollary}

\begin{proof}
See App. \ref{subsec: App, Proof of Cor: Positive Cost of Privacy}.
\end{proof}


\subsection{When does the Cost of Privacy Become Negative?}
\label{subsec: When does the Cost of Privacy Become Negative?}

To complement our results in Cor. \ref{Cor: Positive Cost of Privacy}, we characterize a class of dynamic LQ games with strictly negative costs of privacy. First, for notational ease, we define
$\sigma_A := \Vert A \Vert_2$,
$\hat \sigma_B := \max_{i \in [N]} \Vert B^i \Vert_2$,
$\hat \sigma_Q := \max_{i \in [N]} \Vert Q^i \Vert_2$,
$\hat \sigma_R := \max_{i \in [N]} \Vert R^i \Vert_2$, and $\hat \mu_R := \min_{i \in [N]} \sigma_{\min}(R^{ii})$,
where $\sigma_{\min}(\cdot)$ returns the minimum singular value of a matrix.
Moreover, for each $i, j \in [N]$, we will denote by $Q^{i,j} \in \R^{n_j \times n_j}$ the $(j, j)$-th block of the matrix $Q^i \in \R^{n \times n}$ associated with the state $x_j$ of Agent $j$.

Before proceeding, we introduce the following lemma.

\begin{lemma}
\label{Lemma: f B properties}
Suppose $\sigma_A < 1$. Define:
{
\setlength{\abovedisplayskip}{3pt}
\setlength{\belowdisplayskip}{3pt}
\begin{align}
\label{Eqn: S rho}
    S_\rho &:= \Big\{ \rho \in \Big[ 1, \frac{1}{\sigma_A^2} \Big]: (2 \hat \mu_R \rho + \hat \sigma_R)^2 \geq 4\rho (\rho - 1) \Big\}.
\end{align}
}
$\hspace{-1mm}$and define $f_B: S_\rho \ra \R$ as follows, for each $\rho \in S_\rho$:
{
\setlength{\abovedisplayskip}{3pt}
\setlength{\belowdisplayskip}{3pt}
\begin{align}
\label{Eqn: f B}
    f_B(\rho) &:= \Bigg( \frac{1 - \sigma_A^2 \rho}{N \hat \sigma_Q} \Bigg( \frac{2\rho \hat \mu_R + \hat \sigma_R}{2\rho} \\
    \nonumber
    &\hspace{1cm} - \frac{\sqrt{(2 \rho \hat \mu_R + \hat \sigma_R)^2 - 4\rho (\rho - 1)}}{2\rho} \Bigg) \Bigg)^{1/2}.
\end{align}
}
$\hspace{-1mm}$Then $S_\rho$ is non-empty and compact, $\max_{\rho \in S_\rho} f_B(\rho) > 0$, and
$\arg\max_{\rho \in S_\rho} f_B(\rho) \subset (1, 1/\sigma_A^2)$.
\end{lemma}


\begin{proof}
See App. \ref{subsec: App, Proof of Lemma: f B properties}.
\end{proof}

Our main theorem below establishes that enforcing reference noise alone can induce a \textit{negative} cost of privacy for an Agent $i \in [N]$, indicating that Agent $i$ may simultaneously enforce privacy while improving its cost performance.

\begin{theorem}
\label{Thm: Negative Cost of Privacy}
Suppose Assumptions \ref{Assump: System Observability}-\ref{Assump: Phi and Z matrices are invertible, Infinite-horizon} hold, and that agents may privatize references but not outputs, i.e.,
{
\setlength{\abovedisplayskip}{3pt}
\setlength{\belowdisplayskip}{3pt}
\begin{align}
\label{Eqn: No output privatization noise, for Negative Cost of Privacy}
    \Sigma_v^j &= 0, \ \ \forall \ j \in [N],
\end{align}
}
$\hspace{1mm}$and that, for some 
$\rho^\star \in \arg\max_{\rho \in S_\rho} f_B(\rho)$:
{
\setlength{\abovedisplayskip}{3pt}
\setlength{\belowdisplayskip}{3pt}
\begin{subequations}
\label{Eqn: System assumptions, for Negative Cost of Privacy}
\begin{align}
\label{Eqn: A1, for Negative Cost of Privacy}
    \sigma_A &< 1, \\
\label{Eqn: A2, for Negative Cost of Privacy}
    \hat \sigma_B &< \sqrt{\frac{\hat \mu_R(1 - \sigma_A^2 \rho^\star)(1 - \sigma_A)}{4N \hat \sigma_Q}}, \\
\label{Eqn: A3, for Negative Cost of Privacy}
    \hat \sigma_B &\leq f_B(\rho^\star) \\
\label{Eqn: A4, for Negative Cost of Privacy}
    \hat \sigma_B &\leq \sqrt{\frac{\hat \mu_R(1 - \sigma_A^2 \rho^\star)(1 - \sigma_A)}{4 \sqrt{N} \hat \sigma_Q \big(4 \sqrt{N} \hat \sigma_R \sigma_A + \hat \mu_R (1 + \sigma_A) \big)}}.
\end{align}
\end{subequations}
}
Moreover, suppose that, for some Agent $i$:
{
\setlength{\abovedisplayskip}{3pt}
\setlength{\belowdisplayskip}{3pt}
\begin{subequations}
\label{Eqn: Agent i assumptions, for Negative Cost of Privacy}
\begin{align}
\label{Eqn: A5, Part 1, for Negative Cost of Privacy}
    &Q^i \not\in \symPSD^n, \\
\label{Eqn: A5, Part 2, for Negative Cost of Privacy}
    &\tr(Q^i \Sigma_{\bar w}) < - \frac{64 N \sqrt{n n_u^i} \hat \sigma_B^2 \hat \sigma_Q^2 \hat \sigma_R \hat \sigma_{\bar w}}{\hat \mu_R^2 (1 - \sigma_A)^2} < 0.
\end{align}
\end{subequations}
}
Then the cost of privacy for Agent $i$ is strictly negative, i.e., $\tr\big( (Q^i + F^\top R^i F) \Sigma_{\bar w} \big) < 0$. 
\end{theorem}

We will establish Thm. \ref{Thm: Negative Cost of Privacy} via a sequence of lemmas. 
The first lemma below establishes that bounding $F$ provides an approach for proving Thm. \ref{Thm: Negative Cost of Privacy}.

\begin{lemma}
\label{Lemma: Bounding F}
Under Assumptions \ref{Assump: System Observability}-\ref{Assump: Phi and Z matrices are invertible, Infinite-horizon} and \eqref{Eqn: No output privatization noise, for Negative Cost of Privacy}-\eqref{Eqn: Agent i assumptions, for Negative Cost of Privacy}, if:
{
\setlength{\abovedisplayskip}{3pt}
\setlength{\belowdisplayskip}{3pt}
\begin{align}
\label{Eqn: IR 1, norm of F}
    \Vert F \Vert_2 &< \frac{\sqrt{|\tr(Q^i \Sigma_{\bar w})|}}{(n n_u^i)^{1/4} \sqrt{\hat \sigma_R \hat \sigma_{\bar w}}},
\end{align}
}
$\hspace{-1mm}$then $\tr\big( (Q^i + F^\top R^i F) \Sigma_{\bar w} \big) < 0$.
\end{lemma}

\begin{proof}
See App. \ref{subsec: App, Proof of Lemma: Bounding F}.
\end{proof}

The following lemma in turn establishes a set of sufficient conditions for \eqref{Eqn: IR 1, norm of F}, namely, \eqref{Eqn: IR, all}.

\begin{lemma}
\label{Lemma: Bounding Components of F}
Under Assumptions \ref{Assump: System Observability}-\ref{Assump: Phi and Z matrices are invertible, Infinite-horizon} and \eqref{Eqn: No output privatization noise, for Negative Cost of Privacy}-\eqref{Eqn: Agent i assumptions, for Negative Cost of Privacy}, if:
{
\setlength{\abovedisplayskip}{3pt}
\setlength{\belowdisplayskip}{3pt}
\begin{subequations}
\label{Eqn: IR, all}
\begin{align}
\label{Eqn: IR 3, norm of Pi}
    &\Vert \Pi \Vert_2 \leq \frac{2}{\hat \mu_R}, \\
\label{Eqn: IR 2, norm of Pi, diag B(I - (A + BL) transpose) inv, hat L}
    &\Vert \Pi \ \diag\big\{ B^{i\top} \big( I - (A + BL)^\top \big)^{-1} \big\} \hat L \Vert_2 \leq \frac{1}{2}, \\
\label{Eqn: IR 4 norm of I - (A + BL) transpose inverse}
    &\Vert \big( I - (A + BL)^\top \big)^{-1} \Vert_2 \leq \frac{2}{1 - \sigma_A},
\end{align}
\end{subequations}
}
$\hspace{-1mm}$then \eqref{Eqn: IR 1, norm of F} holds.
\end{lemma}

\begin{proof}
See App. \ref{subsec: App, Proof of Lemma: Bounding Components of F}.
\end{proof}

In the following lemmas, we verify that \eqref{Eqn: IR, all} indeed readily follows from Assumptions \ref{Assump: System Observability}-\ref{Assump: Phi and Z matrices are invertible, Infinite-horizon} and \eqref{Eqn: No output privatization noise, for Negative Cost of Privacy}-\eqref{Eqn: Agent i assumptions, for Negative Cost of Privacy}, thus concluding the proof of Thm. \ref{Thm: Negative Cost of Privacy}. 
For convenience, we define $\hat \sigma_{z,t} := \max_{i \in [N]} \Vert Z_t^i \Vert_2$ and $\hat \sigma_z := \sup_{t \leq T+1} \hat \sigma_{z,t}$.
First, we present conditions under which a bound on $\hat \sigma_z$ can be formulated by leveraging the Riccati equations \eqref{Eqn: Z t i} to recursively bound $\Vert Z_t^i \Vert_2$, for each $i \in [N]$, through all $t \leq T+1$.

\begin{lemma}
\label{Lemma: Bound on Z (sigma z)}
Under Assumption \ref{Assump: Phi and Z matrices are invertible, Infinite-horizon} and \eqref{Eqn: A2, for Negative Cost of Privacy}-\eqref{Eqn: A3, for Negative Cost of Privacy}:
{
\setlength{\abovedisplayskip}{3pt}
\setlength{\belowdisplayskip}{3pt}
\begin{align}
\label{Eqn: Bound on Z (sigma z)}
    \hat \sigma_z &\leq \frac{\hat \sigma_Q}{1 - \sigma_A^2 \rho^\star},
\end{align}
}
where $\rho^\star \in \arg\max_{\rho \in S_\rho} f_B(\rho)$ is as specified in \eqref{Eqn: A3, for Negative Cost of Privacy}.
\end{lemma}

\begin{proof}
See App. \ref{subsec: App, Proof of Bound on Z (sigma z)}.
\end{proof}

Our final lemma below leverages Lemma \ref{Lemma: Bound on Z (sigma z)} to establish that \eqref{Eqn: IR, all} holds under Assumptions \ref{Assump: System Observability}-\ref{Assump: Phi and Z matrices are invertible, Infinite-horizon} and \eqref{Eqn: No output privatization noise, for Negative Cost of Privacy}-\eqref{Eqn: Agent i assumptions, for Negative Cost of Privacy}, and thus completes the proof of Thm. \ref{Thm: Negative Cost of Privacy}.

\begin{lemma}
\label{Lemma: Final}
Under Assumptions \ref{Assump: System Observability}-\ref{Assump: Phi and Z matrices are invertible, Infinite-horizon}, if \eqref{Eqn: No output privatization noise, for Negative Cost of Privacy}-\eqref{Eqn: Agent i assumptions, for Negative Cost of Privacy} hold, then \eqref{Eqn: IR, all} holds.
\end{lemma}

\begin{proof}
See App. \ref{subsec: App, Proof of Lemma: Final}.
\end{proof}

We conclude by describing how Thm. \ref{Thm: Negative Cost of Privacy} can be established from Lemmas \ref{Lemma: Bounding F}-\ref{Lemma: Final}.

\begin{proof}(\textbf{Proof of Thm. \ref{Thm: Negative Cost of Privacy}})
First, Lemma \ref{Lemma: Final} establishes \eqref{Eqn: IR, all} from Assumptions \ref{Assump: System Observability}-\ref{Assump: Phi and Z matrices are invertible, Infinite-horizon} and \eqref{Eqn: No output privatization noise, for Negative Cost of Privacy}-\eqref{Eqn: Agent i assumptions, for Negative Cost of Privacy}.
Then, Lemma \ref{Lemma: Bounding Components of F} establishes \eqref{Eqn: IR 1, norm of F} from \eqref{Eqn: IR, all}.
Finally, Lemma \ref{Lemma: Bounding F} asserts that $\tr\big( (Q^i + F^\top R^i F) \Sigma_{\bar w} \big) < 0$, as desired.
\end{proof}

        
            


\section{Numerical Experiments}
\label{sec: Numerical Experiments}

To corroborate Thm. \ref{Thm: Negative Cost of Privacy}, we computed equilibrium costs associated with the following 2-agent LQG dynamic game $\G_1$, with horizon $T = 100$. We set $A = 0.5 I_4$,
{
\setlength{\abovedisplayskip}{3pt}
\setlength{\belowdisplayskip}{3pt}
\begin{align}
\nonumber
    B^1 = \begin{bmatrix}
        0.01 I_2 \\ O_2
    \end{bmatrix}, \quad
    B^2 = \begin{bmatrix}
        O_2 \\ 0.01 I_2 
    \end{bmatrix},
\end{align}
}
$\hspace{-1.5mm}$$C^1 = C^2 = I_2$, 
$(R^{11}, R^{12}, R^{21}, R^{22}) = (I_2, O_2, O_2, I_2)$, 
$Q^1 = \diag\{-1, -1, 0.1, 0.1\}$, 
$Q^2 = \diag\{0.1, 0.1, -1, -1\}$, 
$\Sigma_w = I_4$,
$\Sigma_{\bar w}^1 = I_2$, 
$\Sigma_{\bar w}^2 = O_2$, and 
$\bar x^1 = (0, 0)$, 
$\bar x^2 = (1, 1)$.
(These parameters satisfy \eqref{Eqn: No output privatization noise, for Negative Cost of Privacy}-\eqref{Eqn: Agent i assumptions, for Negative Cost of Privacy} in the statement of Thm. \ref{Thm: Negative Cost of Privacy}).
The choice of $\Sigma_{\bar w}^1 = I_2$ corresponds to privacy parameters $\delta^1 = 0.01$ and $\epsilon^1 = 1.0$, and adjacency parameter $b^1 = 0.4$ (see Prop. \ref{Prop: Calibrating Privacy and Adjacency Parameters}).
We computed the FBNE strategies $\boldgamma_\steady$ under the CLIS information pattern. We then computed, across $2000$ independently sampled trajectories, the average equilibrium cost $\tilde J^1(x_0, \boldgamma_\steady)$ of Agent 1,
which
was $-6.45$ (resp., $-4.47$) with (resp., without) reference noise, representing an empirical privacy cost of $-1.99$ (see Fig. \ref{Fig: Negative_CoP_distribution___T_100_num_samples_2000.pdf}). 
In comparison, our theory predicted $-1.92$ as the cost of privacy, representing only a $3.5\%$ deviation from our empirical results.
Overall, our results illustrate that it is indeed possible, in some dynamic game settings, for an agent to enforce privacy while improving its cost outcome.

To corroborate Cor. \ref{Cor: Positive Cost of Privacy}, we repeated the same experiment with a 2-agent LQG dynamic game $\G_2$ with nearly the same specification as $\G_1$, except with $Q^1 = \diag\{1, 1, 0.1, 0.1\}$ and $Q^2 = \diag\{0.1, 0.1, 1, 1\}$.
Our results yielded an average equilibrium cost of 8.04 (resp., 6.03) with (resp., without) reference noise, representing an empirical privacy cost of 2.08.
In comparison, our theory predicted $2.01$ as the cost of privacy, representing only a $3.4\%$ deviation from our empirical results.
Our results demonstrate that, when agent objectives are sufficiently aligned, privatization incurs a positive performance cost.


\begin{figure}[t]
    \centering
    \includegraphics[width=0.96\linewidth]{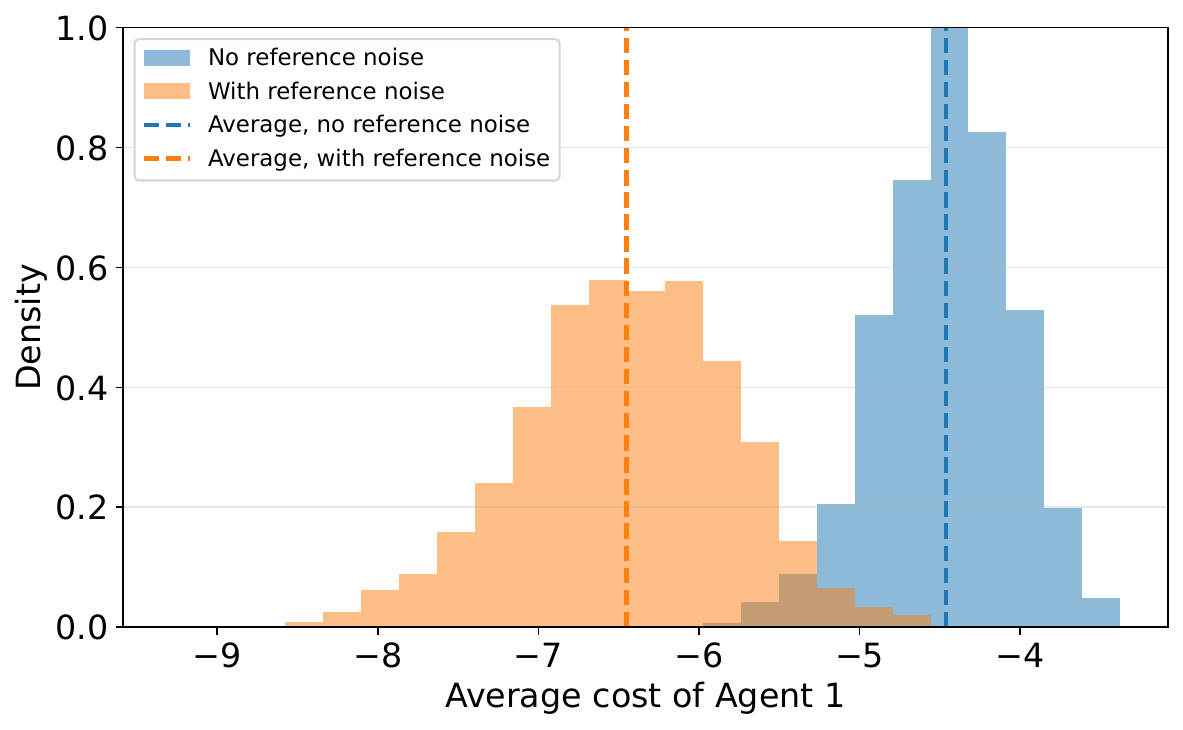}
    \vspace{-3mm}
    \caption{Cost distribution of Agent 1 across $2000$ sampled trajectories, under the dynamics and costs specified by $\G_1$ (see Sec. \ref{sec: Numerical Experiments}).
    Here, Agent 1 simultaneously encodes privacy while lowering its average incurred cost.}
    \vspace{-3mm}
    \label{Fig: Negative_CoP_distribution___T_100_num_samples_2000.pdf}
\end{figure}

\begin{figure}[t]
    \centering
    \includegraphics[width=0.96\linewidth]{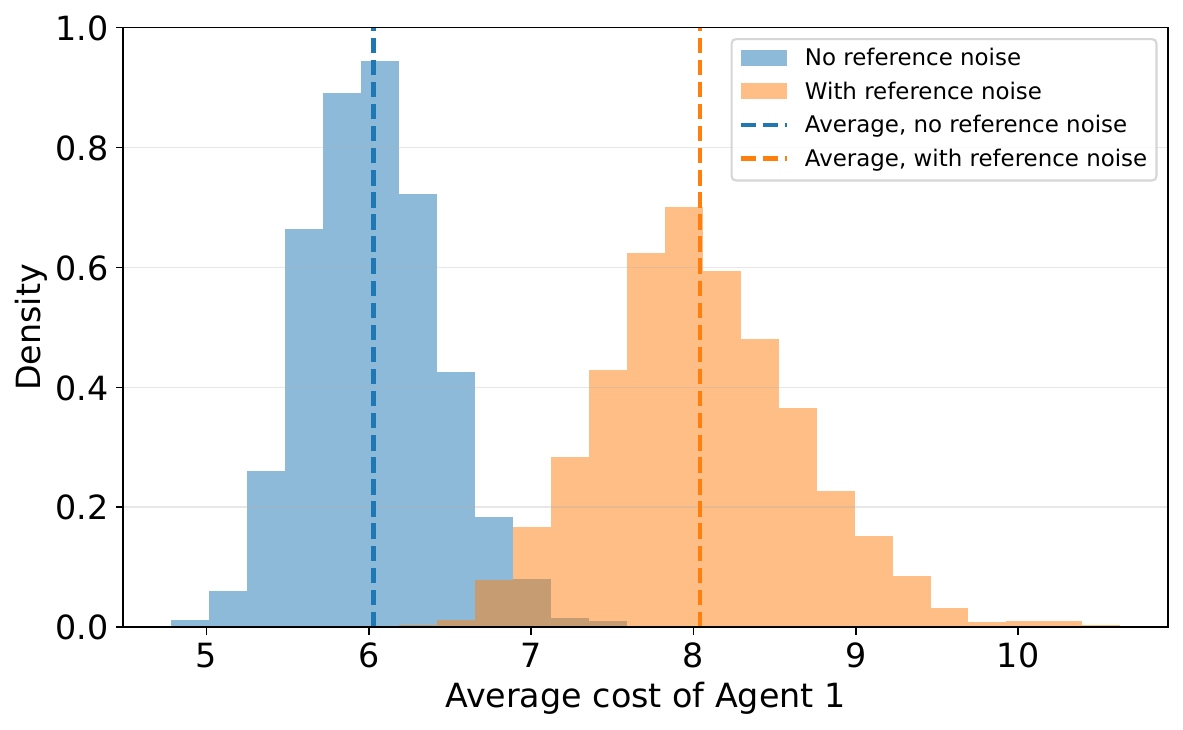}
    \caption{Cost distribution of Agent 1 across $2000$ sampled trajectories, under the dynamics and costs specified by $\G_2$ (see Sec. \ref{sec: Numerical Experiments}).
    Here, Agent 1 incurs a positive cost of privacy.}
    \vspace{-7mm}
    \label{Fig: Positive_CoP_distribution___T_100_num_samples_2000.pdf}
\end{figure}

\section{Conclusion and Future Work}
\label{sec: Conclusion and Future Work}

We studied the cost of differential privacy in stochastic LQ dynamic games with heterogeneous agent objectives. We derived an analytical characterization of each agent's steady-state cost of privacy and established conditions under which privacy noise incurs a positive performance cost. 
Counterintuitively, we also identified a class of games in which appropriately designed reference privatization noise can yield a strictly negative cost of privacy for an agent. 
Future work will investigate cooperative and competitive selection of privacy noise levels, and extend the framework to games with general convex costs and nonlinear constraints.


\printbibliography

@book{basar1998DynamicNoncooperativeGameTheory,
  title={{Dynamic Noncooperative Game Theory}},
  author={Ba{\c{s}}ar, Tamer and Olsder, Geert Jan},
  year={1998},
  publisher={SIAM}
}

@inproceedings{Dwork2006DifferentialPrivacy,
	author = {Dwork, Cynthia},
	title = {{Differential Privacy}},
	booktitle = {International Colloquium on Automata, Languages and Programming, part II (ICALP)},
	year = {2006},
	publisher = {Springer Verlag},
	pages = {1-12},
	journal = {},
	volume = {4052},
}

@Article{Dwork2014AlgorithmicFoundationsOfDifferentialPrivacy,
	author    = {Dwork, Cynthia and Roth, Aaron and others},
	title     = {{The Algorithmic Foundations Of Differential Privacy}},
	journal   = {Foundations and Trends{\textregistered} in Theoretical Computer Science},
	year      = {2014},
	volume    = {9},
	number    = {3--4},
	pages     = {211--407},
	publisher = {Now Publishers, Inc.},
}

@Article{Kasiviswanathan2008OnTheSematicsofDifferentialPrivacy,
	author  = {Kasiviswanathan, Shiva Prasad and Smith, Adam},
	title   = {{On the Semantics of Differential Privacy: A Bayesian Formulation}},
	journal = {arXiv preprint arXiv:0803.3946},
	year    = {2008},
}

@Article{Johnson1991DiscreteandSampledDataStochasticControlProblemswithCompleteandIncompleteStateInformation,
	author="Johnson, A.",
	title="Discrete and sampled-data stochastic control problems with complete and incomplete state information",
	journal="Applied Mathematics and Optimization",
	year="1991",
	volume="24",
	number="1",
	pages="289--316",
}

@inproceedings{LiuLi2025PolicieswithSparseInterAgentDependenciesinDynamicGames,
    author = {Liu, Xinjie and Li, Jingqi and Fotiadis, Filippos and Karabag, Mustafa O. and Milzman, Jesse and Fridovich-Keil, David and Topcu, Uuk},
    title = {{Policies with Sparse Inter-Agent Dependencies in Dynamic Games: A Dynamic Programming Approach}},
    booktitle = {International Conference on Autonomous Agents and Multiagent Systems},
    year = {2025},
}

@INPROCEEDINGS{FridovichKeil2020iLQGames,
  author={Fridovich-Keil, David and Ratner, Ellis and Peters, Lasse and Dragan, Anca and Tomlin, Claire},
  booktitle={IEEE International Conference on Robotics and Automation}, 
  title={{Efficient Iterative Linear-Quadratic Approximations for Nonlinear Multi-Player General-Sum Differential Games}}, 
  year={2020},
  volume={},
  number={},
  pages={1475-1481},
}

@book{Bertsekas1995DynamicProgrammingAndOptimalControl,
author = {Bertsekas, Dimitri P.},
title = {{Dynamic Programming and Optimal Control}},
publisher = {Athena Scientific},
year = {1995},
}

@ARTICLE{Yazdani2023DifferentiallyPrivateLQControl,
  author={Yazdani, Kasra and Jones, Austin and Leahy, Kevin and Hale, Matthew},
  journal={IEEE Transactions on Automatic Control}, 
  title={{Differentially Private LQ Control}}, 
  year={2023},
  volume={68},
  number={2},
  pages={1061-1068},
}

@INPROCEEDINGS{Soudbakhsh2014DelayawareCodesignsforWideAreaControlofPowerGrids,
  author={Soudbakhsh, Damoon and Chakrabortty, Aranya and Annaswamy, Anuradha M.},
  booktitle={IEEE Conference on Decision and Control (CDC)}, 
  title={{Delay-aware Co-designs for Wide-Area Control of Power Grids}}, 
  year={2014},
  volume={},
  number={},
  pages={2493-2498},
}

@INPROCEEDINGS{Singh2016DecentralizedControlofOscillatoryDynamicsinPowerSystemsUsinganExtendedLQR,
  author={Singh, Abhinav and Pal, Bikash},
  booktitle={IEEE Power and Energy Society General Meeting}, 
  title={{Decentralized Control of Oscillatory Dynamics in Power Systems using an Extended LQR}}, 
  year={2016},
  volume={},
  number={},
  pages={1-1},
}

@ARTICLE{McDaniel2009SecurityandPrivacyChallengesintheSmartGrid,
  author={McDaniel, Patrick and McLaughlin, Stephen},
  journal={IEEE Security \& Privacy}, 
  title={{Security and Privacy Challenges in the Smart Grid}}, 
  year={2009},
  volume={7},
  number={3},
  pages={75-77},
}

@ARTICLE{Ndrio2023CoordinatedTransactionSchedulinginMultiAreaElectricityMarkets,
  author={Ndrio, Mariola and Bose, Subhonmesh and Tong, Lang and Guo, Ye},
  journal={IEEE Transactions on Power Systems}, 
  title={{Coordinated Transaction Scheduling in Multi-Area Electricity Markets: Equilibrium and Learning}}, 
  year={2023},
  volume={38},
  number={2},
  pages={996-1008},
}

@ARTICLE{Saad2009CoalitionalGameTheoryforCommunicationNetworks,
  author={Saad, Walid and Han, Zhu and Debbah, Merouane and Hjorungnes, Are and Basar, Tamer},
  journal={IEEE Signal Processing Magazine}, 
  title={{Coalitional Game Theory for Communication Networks}}, 
  year={2009},
  volume={26},
  number={5},
  pages={77-97},
}

@article{BaillonCominetti2008MarkovianTrafficEquilibrium,
author = {Baillon, Jean-Bernard and Cominetti, Roberto},
year = {2008},
pages = {},
title = {{Markovian Traffic Equilibrium}},
journal = {Mathematical Programming},
}

@article{Roughgarden2002HowBadIsSelfishRouting,
  title={{How Bad is Selfish Routing?}},
  author={Roughgarden, Tim and Tardos, {\'E}va},
  journal={Journal of the ACM},
  volume={49},
  number={2},
  pages={236--259},
  year={2002},
  publisher={ACM New York, NY, USA}
}

@article{Ratliff2018APerspectiveonIncentiveDesignChallengesandOpportunities,
author = {Ratliff, Lillian J. and Dong, Roy and Sekar, Shreyas and Fiez, Tanner},
journal = {Annual Review of Control, Robotics, and Autonomous Systems},
number = {1},
pages = {1--34},
title = {{A Perspective on Incentive Design: Challenges and Opportunities}},
volume = {2},
year = {2018}
}

@inproceedings{Nissim2012PrivacyawareMechanismDesign,
author = {Nissim, Kobbi and Orlandi, Claudio and Smorodinsky, Rann},
title = {{Privacy-aware Mechanism Design}},
year = {2012},
booktitle = {ACM Conference on Electronic Commerce},
pages = {774–789},
numpages = {16},
}

@article{Hawkins2026DifferentiallyPrivateFormationControl,
title = {{Differentially Private Formation Control: Privacy and Network Co-design}},
journal = {Systems \& Control Letters},
volume = {213},
pages = {106434},
year = {2026},
author = {Calvin Hawkins and Matthew Hale},
}

@ARTICLE{NyPappas2014DifferentiallyPrivateFiltering,
  author={Le Ny, Jerome and Pappas, George J.},
  journal={IEEE Transactions on Automatic Control}, 
  title={{Differentially Private Filtering}}, 
  year={2014},
  volume={59},
  number={2},
  pages={341-354},
}

@INPROCEEDINGS{HaleJonesLeahy2018PrivacyInFeedback,
  author={Hale, Matthew and Jones, Austin and Leahy, Kevin},
  booktitle={American Control Conference (ACC)}, 
  title={{Privacy in Feedback: The Differentially Private LQG}}, 
  year={2018},
  volume={},
  number={},
  pages={3386-3391},
}

@INPROCEEDINGS{YazdaniHale2020ErrorBoundsandGuidelinesForPrivacyCalibrationinDifferentiallyPrivateKalmanFiltering,
  author={Yazdani, Kasra and Hale, Matthew},
  booktitle={American Control Conference (ACC)}, 
  title={{Error Bounds and Guidelines for Privacy Calibration in Differentially Private Kalman Filtering}}, 
  year={2020},
  volume={},
  number={},
  pages={4423-4428},
}

@InProceedings{chen23,
  title = 	 {{Differential Privacy in Cooperative Multiagent Planning}},
  author =       {Chen, Bo and Hawkins, Calvin and Karabag, Mustafa O. and Neary, Cyrus and Hale, Matthew and Topcu, Ufuk},
  booktitle = 	 {Conference on Uncertainty in Artificial Intelligence},
  pages = 	 {347--357},
  year = 	 {2023},
  volume = 	 {216},  
}

@INPROCEEDINGS{hsu14,
  author={Hsu, Justin and Gaboardi, Marco and Haeberlen, Andreas and Khanna, Sanjeev and Narayan, Arjun and Pierce, Benjamin C. and Roth, Aaron},
  booktitle={IEEE Computer Security Foundations Symposium}, 
  title={Differential Privacy: An Economic Method for Choosing Epsilon}, 
  year={2014},
  volume={},
  number={},
  pages={398-410},
}

\appendix

\subsection{Proof of Prop. \ref{Prop: FBNE under CLIS (LQ)}}
\label{subsec: App, Proof of Prop: FBNE under CLIS (LQ)}


Suppose by induction that, for some $t \in [0, T]$, we have:
\begin{align}
\label{Eqn: Optimal Cost-to-go (V), LQ recursion}
    V_\tau^i(\info_\tau) &= (\hat x_{\tau|\tau} - \tilde \beta_\tau^i)^\top Z_\tau^i (\hat x_{\tau|\tau} - \tilde \beta_\tau^i) + \tilde c_\tau^i, \\
    \nonumber
    &\qquad \forall \ \tau \in [t+1, T+1].
\end{align}
We will prove Prop. \ref{Prop: FBNE under CLIS (LQ)} by establishing that \eqref{Eqn: Optimal Cost-to-go (V), LQ recursion} holds at $t = T+1$, and recursively proving that if \eqref{Eqn: Optimal Cost-to-go (V), LQ recursion} holds for some $t \in [0, T]$, then \eqref{Eqn: Optimal Cost-to-go (V), LQ recursion} also holds at time $t$. 

First, \eqref{Eqn: Optimal Cost-to-go (V), LQ recursion} holds at $t = T+1$, due to the definitions of the maps $W_{T+1}^i(\cdot, \cdot)$ and $V_{T+1}^i(\cdot, \cdot)$, and the quantities $Z_{T+1}^i$, $\tilde \beta_{T+1}^i$, and $\tilde c_{T+1}^i$, respectively defined by \eqref{Eqn: Cost-to-go (W) initialization, at time T+1}, \eqref{Eqn: Optimal Cost-to-go (V), Def}, \eqref{Eqn: Cost-to-go matrix for Dynamic LQ Games (Z) initialization, at time T+1}, \eqref{Eqn: Cost-to-go vector for Dynamic LQ Games (beta) initialization, at time T+1}, and \eqref{Eqn: Cost-to-go bias term for Dynamic LQ Games (c) initialization, at time T+1}.
Next, from \eqref{Eqn: Cost-to-go (W) recursive DP update}:
\begin{align}
\nonumber
    &W_t^i(\info_t, u_t) \\
    = \ &\E_{z_t} \big[ (x_t - \bar x_t)^\top Q^i (x_t - \bar x_t) + u_t^\top R^i u_t + V_{t+1}^i( \info_{t+1}) \big| \info_t \big] \\
\nonumber
    = \ &\E_{z_t} \big[ (x_t - \bar x_t)^\top Q^i (x_t - \bar x_t) + u_t^\top R^i u_t \\
    \nonumber
    &\hspace{1cm} + (
    x_{t+1}
    - \tilde \beta_{t+1}^i)^\top Z_{t+1}^i (x_{t+1} - \tilde \beta_{t+1}^i) + \tilde c_{t+1}^i \big| \info_t \big] \\
\nonumber
    = \ &\E_{z_t} \big[ (x_t - \bar x_t)^\top Q^i (x_t - \bar x_t) + u_t^\top R^i u_t \\ 
    \nonumber
    &\hspace{1cm} + ( \hat x_{t+1|t+1} - \tilde \beta_{t+1}^i)^\top Z_{t+1}^i (\hat x_{t+1|t+1} - \tilde \beta_{t+1}^i) + \\
    \nonumber
    &\hspace{1cm} \tilde c_{t+1}^i \big| \info_t \big] \\
\nonumber
    = \ &\E_{z_t} \big[ (x_t - \bar x_t)^\top Q^i (x_t - \bar x_t) + u_t^\top R^i u_t \\ 
    \nonumber
    &\hspace{1cm} + (A \hat x_{t|t} + B u_t + M_{t+1} e_{t+1} - \tilde \beta_{t+1}^i)^\top Z_{t+1}^i \\
    \nonumber
    &\hspace{1.5cm} (A \hat x_{t|t} + B u_t + M_{t+1} e_{t+1} - \tilde \beta_{t+1}^i) + \tilde c_{t+1}^i \big| \info_t \big] \\
\label{Eqn: Cost-to-go (W) at time t (LQ)}
    = \ &(\hat x_{t|t} - \tilde x_t)^\top Q^i (\hat x_{t|t} - \tilde x_t) + \tr(Q^i \Sigma_{t|t}) \\
    \nonumber
    &\hspace{1cm} + \tilde c_{t+1}^i + \tr(M_{t+1}^\top Z_{t+1}^i M_{t+1} (C \Sigma_{t+1|t} C^\top + \Sigma_v)) \\ 
    \nonumber
    &\hspace{1cm} + u_t^\top R^i u_t + (A \hat x_{t|t} + B u_t - \tilde \beta_{t+1}^i)^\top Z_{t+1}^i \\
    \nonumber
    &\hspace{1.5cm} (A \hat x_{t|t} + B u_t - \tilde \beta_{t+1}^i).
\end{align}
By the definition of Nash equilibria (Def. \ref{Def: FBNE under CLIS}) as tuples of unilaterally optimal strategies, for each $i \in [N]$:
\begin{align}
\nonumber
    0 &= \frac{d}{du_t} W_t^i(\info_t, u_t) \Bigg|_{u_t = \gamma_t^\star(\info_t)} \\
\label{Eqn: First-order Optimality for FBNE under CLIS (LQ)}
    &= (R^{ii} + B^{i\top} Z_{t+1}^i B^i) u_t^i + \sum_{j \in [N] \backslash \{i\}} B^{i\top} Z_{t+1}^i B^j u_t^j \\
    \nonumber
    &\hspace{1cm} + B^{i\top} Z_{t+1}^i (A \hat x_{t|t} - \tilde \beta_{t+1}^i) \Big|_{u_t = \gamma_t^\star(\info_t)} 
\end{align}
Stacking \eqref{Eqn: First-order Optimality for FBNE under CLIS (LQ)} across $i \in [N]$, and applying the definitions of $\Phi_t$, $H_t$, $\tilde h_t$, $L_t$, and $\tilde g_t$, as respectively given by \eqref{Eqn: Phi t}, \eqref{Eqn: H t}, \eqref{Eqn: tilde h t}, \eqref{Eqn: L t}, and \eqref{Eqn: tilde g t}, we obtain:
\begin{align}
\nonumber
    &\Phi_t \gamma_t^\star(\info_t) = - H_t A \hat x_{t|t} + \tilde h_t, \\
\nonumber
    \Ra \ &\gamma_t^\star(\info_t) = - \Phi_t^{-1} H_t A \hat x_{t|t} + \Phi_t^{-1} \tilde h_t = L_t \hat x_{t|t} + \tilde g_t,
\end{align}
which establishes \eqref{Eqn: FBNE under CLIS (LQ)} at time $t$.

It remains to prove that \eqref{Eqn: Optimal Cost-to-go (V), LQ recursion} holds at time $t-1$. 
We proceed by substituting \eqref{Eqn: FBNE under CLIS (LQ)} (with time $t$) into \eqref{Eqn: Cost-to-go (W) at time t (LQ)}, as follows:
\begin{align}
\nonumber
    &\hspace{5mm} V_t^i(\info_t) \\
\nonumber
    &= W_t^i(\info_t, \gamma_t^\star(\info_t)) \\
\nonumber
    &= (\hat x_{t|t} - \tilde x_t)^\top Q^i (\hat x_{t|t} - \tilde x_t) + \tr(Q^i \Sigma_{t|t}) + \tilde c_{t+1}^i \\
    \nonumber
    &\hspace{1cm} + \tr\big( M_{t+1}^\top Z_{t+1}^i M_{t+1} (C \Sigma_{t+1|t} C^\top + \Sigma_v) \big) \\
    \nonumber
    &\hspace{1cm} + (L_t \hat x_{t|t} + \tilde g_t)^\top R^i (L_t \hat x_{t|t} + \tilde g_t) \\
    \nonumber
    &\hspace{1cm} + \big( (A + B L_t) \hat x_{t|t} - \tilde \beta_{t+1}^i \big)^\top Z_{t+1}^i \\
    \nonumber
    &\hspace{1.5cm} \big( (A + B L_t) \hat x_{t|t} - \tilde \beta_{t+1}^i \big) \\
\nonumber
    &= \hat x_{t|t}^\top (Q^i + L_t^\top R^i L_t + (A + B L_t)^\top Z_{t+1}^i (A + B L_t) ) \hat x_{t|t} \\
    \nonumber
    &\hspace{1cm} + 2 \hat x_{t|t}^\top \big( - Q^i \tilde x_t + L_t^\top R^i \tilde g_t \\
    \nonumber
    &\hspace{2.5cm} + (A + B L_t)^\top Z_{t+1}^i (B \tilde g_t - \tilde \beta_{t+1}^i) \big) \\
    \nonumber
    &\hspace{1cm} + \tilde c_{t+1}^i + \tr(Q^i \Sigma_{t|t}) + \tr\big( Z_{t+1}^i (\Sigma_{t+1|t} - \Sigma_{t+1|t+1}) \big) \\ 
    \nonumber
    &\hspace{1cm} + \tilde x_t^\top Q^i \tilde x_t + \tilde g_t^\top R^i \tilde g_t \\\
    \nonumber
    &\hspace{1cm} + (B \tilde g_t - \tilde \beta_{t+1}^i)^\top Z_{t+1}^i (B \tilde g_t - \tilde \beta_{t+1}^i) \\
\nonumber
    &= (\hat x_{t|t} - \tilde \beta_t^i)^\top Z_t^i (\hat x_{t|t} - \tilde \beta_t^i) + \tilde c_t^i,
\end{align}
where the final equality above follows by applying the definitions of $Z_t^i$, $\tilde \beta_t^i$, and $\tilde c_t^i$, given by \eqref{Eqn: Z t i}, \eqref{Eqn: beta t i}, and \eqref{Eqn: c t i}, respectively. 
Thus, \eqref{Eqn: Optimal Cost-to-go (V), LQ recursion} holds at time $t$, and the proof is completed.

\subsection{Proof of Thm. \ref{Thm: Cost of Privacy}}
\label{subsec: App, Proof of Thm. Cost of Privacy}


First, we quantify the impact of output noise privatization on each agent's cost. 
Let $\hat J^i(x_0, \boldgamma_\steady)$ denote the cost of Agent $i$ when the system of $N$ agents applies output privatization but not reference privatization (i.e., $\bar w_t^i = 0$ for each $t \leq T+1, i \in [N]$).
Analogous to the derivations of \cite[(4.12)]{Johnson1991DiscreteandSampledDataStochasticControlProblemswithCompleteandIncompleteStateInformation} and \cite[Thm. 3]{Yazdani2023DifferentiallyPrivateLQControl}, we quantify the cost of output noise privatization by observing that, under zero output noise (i.e., $\Sigma_v = O$), we have $\Sigma_{t|t} = O$ and $\Sigma_{t+1|t} = \Sigma_w$ for each $t \leq T+1$. 
Thus, by contrasting \eqref{Eqn: c t i} as written versus under the setting where $\Sigma_{t|t} = \Sigma_w$ and $\Sigma_{t+1|t} = O$, and taking the steady-state limits $\Sigma_{t+1|t} \ra \Sigma$ and $\Sigma_{t|t} \ra \bar \Sigma$, we obtain that:
\begin{align}
\label{Eqn: Cost of Output Noise Privatization}
    &\hat J^i(x_0, \boldgamma_\steady) - J^i(x_0, \boldgamma_\steady) \\
    \nonumber
    = \ &\tr \big(Z^i \Sigma + (Q^i - Z^i) \bar \Sigma \big) - \tr(Z^i \Sigma_w).
\end{align}
Second, to compute the cost of reference privatization for each agent $i \in [N]$, we will compare the average stage cost at each time $t$ for Agent $i$ under privacy encoding, i.e., $\E\big[(x_t - \tilde x)^\top Q^i (x_t - \tilde x) + \tilde u_t^\top R^i \tilde u_t \big]$ versus without privacy encoding, i.e., the same expression but with $\tilde x_t = \bar x_t$ and $\tilde u_t = \bar u_t$, when all $x_t$ and $u_t$ are given by the steady-state strategies $\boldgamma_\steady$.
To do so, we will leverage the fact that $\tilde x_t = \bar x_t + \bar w_t$ and that (as established below) $\tilde u_t = \bar u_t + F \bar w$.
Concretely, \eqref{Eqn: beta i, Infinite-horizon} implies that, for each agent $j \in [N]$:
{
\small
\begin{align}
\nonumber
    &\hspace{5mm} Z^j \tilde \beta^j \\
    &= Q^j \tilde x - L^\top R^j \tilde g - (A + B L)^\top Z^j (B \tilde g - \tilde \beta^j) \\
\nonumber
    &= Q^j \tilde x - \big( L^\top R^j \Pi + (A + BL)^\top Z^j B \Pi \big) \tilde h + (A + BL)^\top Z^j \tilde \beta^j \\
\nonumber
    &= Q^j \tilde x - \big( L^\top R^j + (A + BL)^\top Z^j B \big) \tilde g + (A + BL)^\top Z^j \tilde \beta^j,
\end{align}
}
where we have applied the fact that $\tilde g := \Pi \tilde h$, and so $\tilde h = \Phi \tilde g$. 
Rearranging terms, we obtain, for each $j \in [N]$:
{
\small
\begin{align}
\nonumber
    &\big( L^\top R^j + (A + BL)^\top Z^j B \big) \tilde g + \big( I - (A + BL)^\top \big) Z^j \tilde \beta_j = Q^j \tilde x.
\end{align}
}
$\hspace{-1mm}$Stacking vertically across $j \in [N]$, and applying the definitions of $\hat L$, $\hat Z$, and $\hat Q$ from \eqref{Eqn: Notation for Cost of Privacy}, we obtain $\hat L \tilde g + \hat Z \tilde \beta = \hat Q \tilde x$. 
Moreover, since $\tilde g = \Pi \tilde h$, we have $\tilde g^i = \sum_{j \in [N]} \Pi^{ij} \tilde h^j = \sum_{j \in [N]} \Pi^{ij} B^{j\top} Z^j \tilde \beta^j$ for each $i \in [N]$, i.e., $\tilde g = \hat \Pi \tilde \beta$. 
Combining these two equations, we have:
\begin{align}
\nonumber
    \tilde g &= \hat \Pi \tilde \beta = \hat \Pi \hat Z^{-1} (- \hat L \tilde g + \hat Q \tilde x) \\
\nonumber
    &= - \hat \Pi \hat Z^{-1} \hat L \tilde g + \hat \Pi \hat Z^{-1} \hat Q \tilde x \\
\nonumber
    \Ra \ \tilde g &= (I + \hat \Pi \hat Z^{-1} \hat L)^{-1} \hat \Pi \hat Z^{-1} \hat Q \tilde x \\
\nonumber
    &= (I + \hat \Pi \hat Z^{-1} \hat L)^{-1} \hat \Pi \hat Z^{-1} \hat Q (\bar x + \bar w) \\
    &= \bar g + F \bar w.
\end{align}
Thus, under $\boldgamma_\steady$, for any state $x$, and any $i \in [N]$ and $t \leq T+1$, we have:
{
\small
\begin{align}
\nonumber
    &\E\big[(x_t - \tilde x)^\top Q^i (x_t - \tilde x) + \tilde u_t^\top R^i \tilde u_t \big] \\
\nonumber
    = \ &\E\big[(x_t - \bar x - \bar w)^\top Q^i (x_t - \bar x - \bar w) \\
    \nonumber
    &\hspace{1cm} + (u_t + F \bar w)^\top R^i (u_t + F \bar w) \big] \\
\nonumber
    = \ &\E\big[(x_t - \bar x)^\top Q^i (x_t - \bar x) + u_t^\top R^i u_t \big] + \tr\big( (Q^i + F^\top R^i F) \Sigma_{\bar w} \big),
\end{align}
}
which establishes that each Agent $i$'s cost of reference noise privatization is:
\begin{align}
\label{Eqn: Cost of Reference Noise Privatization}
    \tilde J^i(x_0, \boldgamma_\steady) - \hat J^i(x_0, \boldgamma_\steady) &=  \tr\big( (Q^i + F^\top R^i F) \Sigma_{\bar w} \big).
\end{align}
Combining \eqref{Eqn: Cost of Output Noise Privatization} and \eqref{Eqn: Cost of Reference Noise Privatization} gives
\eqref{Eqn: Cost of Privacy} and concludes the proof.

\subsection{Proof of Cor. \ref{Cor: Positive Cost of Privacy}}
\label{subsec: App, Proof of Cor: Positive Cost of Privacy}




Since
$Q^i + F^\top R^i F$ is symmetric positive semidefinite when $Q^i$ is symmetric positive semidefinite, and since $\Sigma_{\bar w}$ is symmetric positive definite as well, we have $\tr\big( (Q^i + F^\top R^i F) \Sigma_{\bar w} \big) \geq 0$. We thus conclude that $\Delta J^i(x_0, \boldgamma_\steady) \geq 0$ for each $i \in [N]$, as claimed.

\subsection{Proof of Lemma \ref{Lemma: f B properties}}
\label{subsec: App, Proof of Lemma: f B properties}


First, $S_\rho$ is the pre-image of the closed set with respect to the continuous map $\rho \mapsto (2\rho \hat \mu_R + \hat \sigma_R)^2 - 4\rho(\rho - 1)$, and is thus closed. 
Moreover, by definition, $S_\rho \subset \big[1, 1/\sigma_A^2 \big]$, and is thus bounded. 
Thus, by the Heine-Borel Theorem, $S_\rho$ is compact, as claimed.

Next, since the map $\rho \mapsto 4\rho(\rho-1)$ returns $0$ at $\rho = 1$, and is strictly increasing and continuous over $[1, \infty)$, there exists some $\bar\rho \in \big(1, 1/\sigma_A^2\big)$ such that $0 < 4 \bar\rho (\bar \rho - 1) < (2 \bar \rho \hat \mu_R + \hat \sigma_R)^2$, and thus $f_B(\bar \rho) > 0$. Thus, $\max_{\rho \in S_\rho} f_B(\rho) > 0$, as claimed.

Finally, since $f_B(\rho) = 0$ at $\rho = 1$ and $\rho = 1/\sigma_A^2$, we find that $\arg\max_{\rho \in S_\rho} f_B(\rho) \subset (1, 1/\sigma_A^2)$, as claimed.

\subsection{Proof of Lemma \ref{Lemma: Bounding F}}
\label{subsec: App, Proof of Lemma: Bounding F}

We have:
{
\setlength{\abovedisplayskip}{3pt}
\setlength{\belowdisplayskip}{3pt}
\begin{align}
\nonumber
    &\hspace{5mm} \tr\big( (Q^i + F^\top R^i F) \Sigma_{\bar w} \big) \\
\nonumber
    &\leq \tr(Q^i \Sigma_{\bar w}) + \Vert F^\top R^i F \Vert_F \Vert \Sigma_{\bar w} \Vert_F \\
\nonumber
    &\leq \tr(Q^i \Sigma_{\bar w}) + \Vert F \Vert_2^2 \Vert R^i \Vert_F \Vert \Sigma_{\bar w} \Vert_F \\
\nonumber
    &\leq \tr(Q^i \Sigma_{\bar w}) + \frac{|\tr(Q^i \Sigma_{\bar w})|}{\sqrt{n n_u^i} \hat \sigma_R \hat \sigma_{\bar w}} \cdot \sqrt{n_u^i} \Vert R^i \Vert_2 \cdot \sqrt{n} \Vert \Sigma_{\bar w} \Vert_2, \\
\nonumber
    &< 0,
\end{align}
}
$\hspace{-1mm}$where the first two inequalities above follow by applying the Cauchy-Schwarz inequality to the Frobenius inner product, and applying the identity that, for any matrices $X, Y \in \R^{n_a \times n_b}$, we have $\Vert X^\top Y \Vert_F \leq \Vert X \Vert_2 \Vert Y \Vert_F$ and $\Vert X \Vert_F \leq \sqrt{\min\{n_a, n_b\}} \Vert X \Vert_2$.

\subsection{Proof of Lemma \ref{Lemma: Bounding Components of F}}
\label{subsec: App, Proof of Lemma: Bounding Components of F}

First, from the definition of $F$ in \eqref{Eqn: F}, we have:
{
\setlength{\abovedisplayskip}{3pt}
\setlength{\belowdisplayskip}{3pt}
\begin{align}
\nonumber
    F &= (I + \hat \Pi \hat Z^{-1} \hat L)^{-1} \hat \Pi \hat Z^{-1} \hat Q \\
\label{Eqn: F, alternative expression}
    &= \big( I + \Pi \ \diag\big\{ B^{i\top} \big( I - (A + BL)^\top \big)^{-1} \big\} \hat L \big)^{-1} \\
\nonumber
    &\hspace{1cm} \cdot \Pi \ \diag\big\{ B^{i\top} \big( I - (A + BL)^\top \big)^{-1} \big\} \hat Q,
\end{align}
}
$\hspace{-1mm}$where \eqref{Eqn: F, alternative expression} arises by applying standard matrix manipulations to \eqref{Eqn: hat Z}-\eqref{Eqn: hat Pi} to establish that $\hat \Pi \hat Z^{-1} = \Pi \ \diag\big\{ B^{i\top} \big( I - (A + BL)^\top \big)^{-1} \big\}$.
From \eqref{Eqn: IR 2, norm of Pi, diag B(I - (A + BL) transpose) inv, hat L}, we find that the singular values of $I + \hat \Pi \hat Z^{-1} \hat L$ must be no less than $1 - 1/2 = 1/2$, and thus:
{
\setlength{\abovedisplayskip}{3pt}
\setlength{\belowdisplayskip}{3pt}
\begin{align}
\label{Eqn: R1, IR 2, norm of Pi, diag B(I - (A + BL) transpose) inv, hat L}
    &\Vert \big( I + \Pi \ \diag\big\{ B^{i\top} \big( I - (A + BL)^\top \big)^{-1} \big\} \hat L \big)^{-1} \Vert_2 \leq 2.
\end{align}
}
$\hspace{-1mm}$Applying the Cauchy-Schwarz inequality to \eqref{Eqn: F, alternative expression} and applying \eqref{Eqn: R1, IR 2, norm of Pi, diag B(I - (A + BL) transpose) inv, hat L}, \eqref{Eqn: IR 3, norm of Pi}, and \eqref{Eqn: IR 4 norm of I - (A + BL) transpose inverse}, we find that:
{
\setlength{\abovedisplayskip}{3pt}
\setlength{\belowdisplayskip}{3pt}
\begin{align}
\nonumber
    \Vert F \Vert_2 &= \Vert \big( I + \Pi \ \diag\big\{ B^{i\top} \big( I - (A + BL)^\top \big)^{-1} \big\} \hat L \big)^{-1} \Vert_2 \\
\nonumber
    &\hspace{5mm} \cdot \Vert \Pi \Vert_2 \ \Vert \diag\big\{ B^{i\top} \big( I - (A + BL)^\top \big)^{-1} \big\} \Vert_2 \Vert \hat Q \Vert_2 \\
\nonumber
    &\leq 2 \cdot \frac{2}{\hat \mu_R} \cdot \hat \sigma_B \cdot \frac{2}{1 - \sigma_A} \cdot \sqrt{N} \hat \sigma_Q \\
\nonumber
    &= \frac{8\sqrt{N} \hat \sigma_B \hat \sigma_Q}{\hat \mu_R (1 - \sigma_A)} \\
\nonumber
    &< \frac{\sqrt{|\tr(Q^i \Sigma_{\bar w})|}}{(n n_u^i)^{1/4} \sqrt{\hat \sigma_R \hat \sigma_{\bar w}}},
\end{align}
}
$\hspace{-1mm}$which confirms \eqref{Eqn: IR 1, norm of F}.
Note that we have used the fact that, from \eqref{Eqn: hat Q}, $\Vert \hat Q \Vert_2 \leq \sqrt{\sum_{i=1}^N \Vert Q^i \Vert_2^2} \leq \sqrt{N} \hat \sigma_Q$, and that the final inequality follows from \eqref{Eqn: A5, Part 2, for Negative Cost of Privacy}.

\subsection{Proof of Lemma \ref{Lemma: Bound on Z (sigma z)}}
\label{subsec: App, Proof of Bound on Z (sigma z)}

We will establish \eqref{Eqn: Bound on Z (sigma z)} by establishing, via backward induction over \eqref{Eqn: Z t i} from $t = T+1$, that:
{
\setlength{\abovedisplayskip}{3pt}
\setlength{\belowdisplayskip}{3pt}
\begin{align}
\label{Eqn: Uniform Bound on Z t (sigma z,t)}
    \hat \sigma_{z,t} &\leq \frac{\hat \sigma_Q}{1 - \sigma_A^2 \rho^\star}, \hspace{5mm} \forall t \leq T+1.
\end{align}
}
$\hspace{-1mm}$Since $Z_t^i \ra Z^i$ for each $i \in [N]$ as $t \ra - \infty$, and the norm map $X \mapsto \Vert X \Vert_2$ is continuous, \eqref{Eqn: Uniform Bound on Z t (sigma z,t)} implies \eqref{Eqn: Bound on Z (sigma z)}.

First, 
since $Z_{T+1}^i = Q^i$ for each $i \in [N]$:
{
\setlength{\abovedisplayskip}{3pt}
\setlength{\belowdisplayskip}{3pt}
\begin{align}
\nonumber
    \hat \sigma_{z,T+1} &= 
    \hat \sigma_Q \leq \frac{\hat \sigma_Q}{1 - \sigma_A^2 \rho^\star}.
\end{align}
}
$\hspace{-1mm}$Next, suppose that for some $t \leq T$:
{
\setlength{\abovedisplayskip}{3pt}
\setlength{\belowdisplayskip}{3pt}
\begin{align}
\label{Eqn: sigma z t+1 bound, Induction hypothesis}
    \hat \sigma_{z,t+1} &\leq \frac{\hat \sigma_Q}{1 - \sigma_A^2 \rho^\star}.
\end{align}
}
$\hspace{-1.5mm}$From \eqref{Eqn: A3, for Negative Cost of Privacy}, we have $\hat \sigma_B \leq f_B(\rho^\star)$.
By applying the definition of $f_B(\cdot)$ from \eqref{Eqn: f B}, rearranging terms, and applying the induction hypothesis, we have:
{
\small
\setlength{\abovedisplayskip}{3pt}
\setlength{\belowdisplayskip}{3pt}
\begin{align}
\nonumber
    &\hspace{5mm} N \hat \sigma_B^2 \hat \sigma_{z,t+1} \\
\nonumber
    &\leq \frac{N \hat \sigma_B^2 \hat \sigma_Q}{1 - \sigma_A^2 \rho^\star} \\
\label{Eqn: Smaller real root w.r.t. map defined relative to f B}
    &\leq \frac{2\rho^\star \hat \mu_R + \hat \sigma_R}{2 \rho^\star} - \frac{\sqrt{(2\rho^\star \hat \mu_R + \hat \sigma_R)^2 - 4\rho^\star(\rho^\star - 1)}}{2\rho^\star}.
\end{align}
}
$\hspace{-1mm}$Since the map $\alpha \mapsto \rho^\star \alpha^2 - (2\rho^\star \hat \mu_R + \hat \sigma_R) \alpha + (\rho^\star - 1) \hat \mu_R^2$ is strictly convex with two real roots, the smaller of which is given by \eqref{Eqn: Smaller real root w.r.t. map defined relative to f B}, the above inequalities imply that this map evaluates to a non-negative value when $\alpha = N \hat \sigma_B^2 \hat \sigma_{z,t+1}$:
{
\setlength{\abovedisplayskip}{3pt}
\setlength{\belowdisplayskip}{3pt}
\begin{align}
\nonumber
    &\rho^\star (N \hat \sigma_B^2 \hat \sigma_{z,t+1})^2 - (2\rho^\star \hat \mu_R + \hat \sigma_R) N \hat \sigma_B^2 \hat \sigma_{z,t+1} \\
    \nonumber
    &\hspace{2cm} + (\rho^\star - 1) \hat \mu_R^2 \geq 0.
\end{align}
}
$\hspace{-1.5mm}$Rearranging terms, we obtain:
{
\setlength{\abovedisplayskip}{3pt}
\setlength{\belowdisplayskip}{3pt}
\begin{align}
\label{Eqn: Fraction bounded by rho star}
    \frac{N \hat \sigma_B^2 \hat \sigma_{z,t+1} \hat \sigma_R + \hat \mu_R^2}{(\hat \mu_R - N \hat \sigma_B^2 \hat \sigma_{z,t+1})^2} \leq \rho^\star.
\end{align}
}
$\hspace{-1.5mm}$We note in particular that $\hat \mu_R - N \hat \sigma_B^2 \hat \sigma_{z,t+1} > 0$, as implied by \eqref{Eqn: A2, for Negative Cost of Privacy} and \eqref{Eqn: sigma z t+1 bound, Induction hypothesis}:
{
\setlength{\abovedisplayskip}{3pt}
\setlength{\belowdisplayskip}{3pt}
\begin{align}
\nonumber
    \hat \mu_R - N \hat \sigma_B^2 \hat \sigma_{z,t+1} &\geq \hat \mu_R (1 - \sigma_A) - \frac{N \hat \sigma_B^2 \hat \sigma_Q}{1 - \sigma_A^2 \rho^\star} \\
    &> \hat \mu_R (1 - \sigma_A) - \frac{4N \hat \sigma_B^2 \hat \sigma_Q}{1 - \sigma_A^2 \rho^\star} > 0.
\end{align}
}
We now leverage \eqref{Eqn: Fraction bounded by rho star} to bound $\hat \sigma_{z,t}$ through \eqref{Eqn: Z t i} and complete the inductive proof.
First, from \eqref{Eqn: Z t i} and \eqref{Eqn: L t}:
{
\setlength{\abovedisplayskip}{3pt}
\setlength{\belowdisplayskip}{3pt}
\begin{align}
\nonumber
    Z_t^i &= Q^i + L_t^\top R^i L_t + (A + B L_t)^\top Z_{t+1}^i (A + B L_t) \\
\label{Eqn: Expansion of Z t i}
    &= Q^i + A^\top H_t^\top \Phi_t^{-1} R^i \Phi_t^{-1} H_t A \\
\nonumber
    &\hspace{5mm} + A^\top (I - H_t^\top \Phi_t^{-1\top} B^\top) Z_{t+1}^i (I - B \Phi_t^{-1} H_t) A.
\end{align}
}
$\hspace{-1mm}$Thus, to bound $\Vert Z_t^i \Vert_2$, we must establish bounds over $\Vert H_t \Vert_2$ and $\Vert \Phi_t^{-1} \Vert_2$.
First, by applying the Cauchy-Schwarz inequality to \eqref{Eqn: H t}, we have:
{
\setlength{\abovedisplayskip}{3pt}
\setlength{\belowdisplayskip}{3pt}
\begin{align}
\label{Eqn: H t, bound}
    \Vert H_t \Vert_2 &\leq \sqrt{N} \hat \sigma_B \hat \sigma_{z,t+1}.
\end{align}
}
$\hspace{-1mm}$Next, by definition of $\Phi_t$ in \eqref{Eqn: Phi t}, for each $v \in \R^m$ with $\Vert v \Vert_2 = 1$, we have:
{
\setlength{\abovedisplayskip}{3pt}
\setlength{\belowdisplayskip}{3pt}
\begin{align}
\nonumber
    \Vert \Phi_t v \Vert_2 &= \Vert \diag\{R^{ii}: i \in [N]\} v + H_t B v \Vert_2 \\
\nonumber
    &\geq \Vert \diag\{R^{ii}: i \in [N]\} v \Vert_2 - \Vert H_t B v \Vert_2 \\
\label{Eqn: Phi t, lower bound on singular values}
    &\geq \hat \mu_R \Vert v \Vert_2 - N \hat \sigma_B^2 \hat \sigma_{z,t+1} \Vert v \Vert_2.
\end{align}
}
$\hspace{-1.5mm}$Thus, the singular values of $\Phi_t$ are no less than $\hat \mu_R - N \hat \sigma_B^2 \hat \sigma_{z,t+1}$, so:
{
\setlength{\abovedisplayskip}{1pt}
\setlength{\belowdisplayskip}{3pt}
\begin{align}
\label{Eqn: Phi t inverse, bound}
    \Vert \Phi_t^{-1} \Vert_2 \leq &\frac{1}{\hat \mu_R - N \hat \sigma_B^2 \hat \sigma_{z,t+1}}.
\end{align}
}
$\hspace{-1mm}$From \eqref{Eqn: Expansion of Z t i} and the bounds \eqref{Eqn: H t, bound} and \eqref{Eqn: Phi t inverse, bound}, we thus obtain after some algebraic manipulations that, for each $i \in [N]$:
{
\setlength{\abovedisplayskip}{3pt}
\setlength{\belowdisplayskip}{3pt}
\begin{align}
\nonumber
    \Vert Z_t^i \Vert_2 &\leq \hat \sigma_Q + \frac{\sigma_A^2 \cdot (\sqrt{N} \hat \sigma_B \hat \sigma_{z,t+1})^2}{(\hat \mu_R - N \hat \sigma_B^2 \hat \sigma_{z,t+1})^2} \hat \sigma_R \\
    \nonumber
    &\hspace{5mm} + \sigma_A^2 \left( 1 + \frac{N \hat \sigma_B^2 \hat \sigma_{z,t+1}}{\hat \mu_R - N \hat \sigma_B^2 \hat \sigma_{z,t+1}} \right)^2 \hat \sigma_{z,t+1} \\
\nonumber
    &= \hat \sigma_Q + \hat \sigma_{z,t+1} \sigma_A^2 \cdot \frac{N \hat \sigma_B^2 \hat \sigma_{z,t+1} \hat \sigma_R + \hat \mu_R^2}{(\hat \mu_R - N \hat \sigma_B^2 \hat \sigma_{z,t+1})^2} \\
\nonumber
    &\leq \hat \sigma_Q + \frac{\hat \sigma_Q}{1 - \sigma_A^2 \rho^\star} \cdot \sigma_A^2 \rho^\star \\
\nonumber
    &= \frac{\hat \sigma_Q}{1 - \sigma_A^2 \rho^\star}
\end{align}
}
$\hspace{-1mm}$Maximizing over all $i \in [N]$ completes the induction step, and thus completes the proof.

\subsection{Proof of Lemma \ref{Lemma: Final}}
\label{subsec: App, Proof of Lemma: Final}

First, we prove \eqref{Eqn: IR 3, norm of Pi}. 
From
\eqref{Eqn: Bound on Z (sigma z)} in Lemma \ref{Lemma: Bound on Z (sigma z)} and \eqref{Eqn: A2, for Negative Cost of Privacy}:
{
\setlength{\abovedisplayskip}{3pt}
\setlength{\belowdisplayskip}{3pt}
\begin{align}
    N \hat \sigma_B^2 \hat \sigma_z &\leq \frac{N \hat \sigma_B^2 \hat \sigma_Q}{1 - \sigma_A^2 \rho^\star} \leq \frac{1}{4} (1 - \sigma_A) \hat \mu_R < \frac{1}{2} \hat \mu_R.
\end{align}
}
Thus, following the argument from \eqref{Eqn: Phi t, lower bound on singular values}, the singular values of $\Phi_t$ (and thus, those of $\Phi$) are no less than $\hat \mu_R - N \hat \sigma_B^2 \hat \sigma_{z,t+1} \geq \frac{1}{2} \hat \mu_R$, from which \eqref{Eqn: IR 3, norm of Pi} follows.

Second, to prove \eqref{Eqn: IR 4 norm of I - (A + BL) transpose inverse}, we first apply the Cauchy-Schwarz inequality to \eqref{Eqn: H, Infinite-Horizon} to obtain $\Vert H \Vert_2 \leq \sqrt{N} \hat \sigma_B \hat \sigma_z$. Then:
{
\setlength{\abovedisplayskip}{3pt}
\setlength{\belowdisplayskip}{3pt}
\begin{align}
\nonumber
    \Vert BL \Vert_2 &= \Vert B \Pi H A \Vert_2 \leq \Vert B \Vert_2 \Vert \Pi \Vert_2 \Vert H \Vert_2 \sigma_A \\
\label{Eqn: Bound on BL, Step 1}
    &\leq \sqrt{N} \hat \sigma_B \cdot \frac{2}{\hat \mu_R} \cdot \sqrt{N} \hat \sigma_B \hat \sigma_z \cdot \sigma_A \\
\label{Eqn: Bound on BL, Step 2}
    &\leq \sqrt{N} \hat \sigma_B \cdot \frac{2}{\hat \mu_R} \cdot \sqrt{N} \hat \sigma_B \cdot \frac{\hat \sigma_Q}{1 - \sigma_A^2 \rho^\star} \cdot \sigma_A \\
\nonumber
    &= \frac{2N \hat \sigma_B^2 \hat \sigma_Q \sigma_A}{\hat \mu_R (1 - \sigma_A^2 \rho^\star)} \\
\label{Eqn: Bound on BL, Step 3}
    &\leq \frac{1}{2}(1 - \sigma_A),
\end{align}
}
$\hspace{-1mm}$where \eqref{Eqn: Bound on BL, Step 1} follows by applying \eqref{Eqn: IR 3, norm of Pi} and the bound $\Vert H \Vert_2 \leq \sqrt{N} \hat \sigma_B \hat \sigma_z$, \eqref{Eqn: Bound on BL, Step 2} follows from \eqref{Eqn: Bound on Z (sigma z)}, and \eqref{Eqn: Bound on BL, Step 3} follows from \eqref{Eqn: A1, for Negative Cost of Privacy} and \eqref{Eqn: A2, for Negative Cost of Privacy}. 
Thus, from the triangle inequality, $\Vert A + BL \Vert_2 \leq \sigma_A + \frac{1}{2}(1 - \sigma_A) = \frac{1}{2} (1 + \sigma_A)$.
It then follows that the singular values of $I - (A + BL)$ must be no less than $1 - \frac{1}{2}(1 + \sigma_A) = \frac{1}{2}(1 - \sigma_A)$, from which \eqref{Eqn: IR 4 norm of I - (A + BL) transpose inverse} follows.

Finally, to establish \eqref{Eqn: IR 2, norm of Pi, diag B(I - (A + BL) transpose) inv, hat L}, we first bound $L$. Similar to the derivation of \eqref{Eqn: Bound on BL, Step 2}:
{
\setlength{\abovedisplayskip}{1pt}
\setlength{\belowdisplayskip}{3pt}
\begin{align}
\label{Eqn: Bound on L}
    \Vert L \Vert_2 &\leq \frac{2\sqrt{N} \hat \sigma_B \hat \sigma_Q \sigma_A}{\hat \mu_R (1 - \sigma_A^2 \rho^\star)}. 
\end{align}
}
$\hspace{-1.5mm}$Also, from \eqref{Eqn: Bound on BL, Step 3} and the ensuing discussion, $\Vert A + BL \Vert_2 \leq \frac{1}{2}(1 + \sigma_A)$.
Thus, via \eqref{Eqn: hat L}, \eqref{Eqn: Bound on L}, and \eqref{Eqn: Bound on Z (sigma z)} (from Lemma \ref{Lemma: Bound on Z (sigma z)}):
{
\small
\begin{align}
\nonumber
    \Vert \hat L \Vert_2 &\leq \sqrt{N} (\Vert L \Vert_2 + \Vert A + BL \Vert_2 \hat \sigma_z \hat \sigma_B) \\
\nonumber
    &\leq \sqrt{N} \Bigg(\frac{2\sqrt{N} \hat \sigma_B \hat \sigma_Q \sigma_A}{\hat \mu_R (1 - \sigma_A^2 \rho^\star)}
    + \frac{1}{2}(1 + \sigma_A) \cdot \frac{\hat \sigma_Q}{1 - \sigma_A^2 \rho^\star} \cdot \hat \sigma_B \Bigg) \\
\label{Eqn: Bound on hat L}
    &= \frac{\sqrt{N} \hat \sigma_Q \hat \sigma_B (4 \sqrt{N} \sigma_A \hat \sigma_R + \hat \mu_R (1 + \sigma_A))}{2 \hat \mu_R (1 - \sigma_A^2 \rho^\star)}.
\end{align}
}
$\hspace{-1mm}$We can now prove \eqref{Eqn: IR 2, norm of Pi, diag B(I - (A + BL) transpose) inv, hat L} as follows:
{
\small
\begin{align}
\nonumber
    &\Vert \Pi \ \diag\big\{ B^{i\top} \big( I - (A + BL)^\top \big)^{-1} \big\} \hat L \Vert_2 \\
\nonumber
    \leq \ &\Vert \Pi \Vert_2 \cdot \hat \sigma_B \Vert (I - (A + BL)^\top)^{-1} \Vert_2 \cdot \Vert \hat L \Vert_2 \\
\label{Eqn: Proving IR 2, norm of Pi, diag B(I - (A + BL) transpose) inv, hat L, Step 1}
    \leq \ &\frac{2}{\hat \mu_R} \cdot \hat \sigma_B \cdot \frac{2}{1 - \sigma_A} \cdot \frac{\sqrt{N} \hat \sigma_Q \hat \sigma_B (4 \sqrt{N} \sigma_A \hat \sigma_R + \hat \mu_R (1 + \sigma_A))}{2 \hat \mu_R (1 - \sigma_A^2 \rho^\star)} \\
\nonumber
    \leq \ &\hat \sigma_B^2 \cdot \frac{2\sqrt{N} \hat \sigma_Q (4 \sqrt{N} \sigma_A \hat \sigma_R + \hat \mu_R (1 + \sigma_A))}{\hat \mu_R^2 (1 - \sigma_A)(1 - \sigma_A^2 \rho^\star)} \\
\label{Eqn: Proving IR 2, norm of Pi, diag B(I - (A + BL) transpose) inv, hat L, Step 2}
    \leq \ &\frac{1}{2},
\end{align}
}
\hspace{-1.5mm}where \eqref{Eqn: Proving IR 2, norm of Pi, diag B(I - (A + BL) transpose) inv, hat L, Step 1} follows from applying \eqref{Eqn: IR 3, norm of Pi}, \eqref{Eqn: IR 4 norm of I - (A + BL) transpose inverse}, and \eqref{Eqn: Bound on hat L}, while \eqref{Eqn: Proving IR 2, norm of Pi, diag B(I - (A + BL) transpose) inv, hat L, Step 2} follows from \eqref{Eqn: A4, for Negative Cost of Privacy}.









\end{document}